\documentclass[12pt]{amsart}
\usepackage{amsmath,amsthm,amssymb}
\usepackage[pdftex]{graphicx}
\usepackage{tikz}
\usepackage{comment}

\theoremstyle{definition}
\newtheorem{dfn}{Definition}[section]
\newtheorem{ex}[dfn]{Example}

\newtheorem{rem}[dfn]{Remark}

\newtheorem*{acknowledgement}{Acknowledgement}

\theoremstyle{plain}
\newtheorem{thm}[dfn]{Theorem}
\newtheorem{prop}[dfn]{Proposition}

\newtheorem{lem}[dfn]{Lemma}

\title[Cluster-cyclic conditions]{Cluster-cyclic condition of skew-symmetrizable matrices of rank 3 via the Markov constant}
\author{Ryota Akagi}
\email{ryota.akagi.e6@math.nagoya-u.ac.jp}
\begin{document}
\maketitle
\begin{abstract}
In this paper, we consider mutations of skew-symmetrizable matrices of rank 3. Every skew-symmetrizable matrix corresponds to a weighted quiver, and we study the conditions when this quiver is always cyclic after applying mutations. In this study, the Markov constant has an essential meaning. It has already appeared in some previous works for skew-symmetric matrices.
\end{abstract}

\section{Introduction}
\subsection{Background}
The cluster algebra theorey was introduced by Fomin and Zelevinsky in \cite{FZ02}. The main objects are {\em mutations}, which are transformations of matrices and rational fractions called cluster variables.
In this paper, we consider mutations of skew-symmetrizable matrices of rank 3.
The matrix $B \in \mathrm{M}_3(\mathbb{R})$ is called {\em skew-symmetrizable} if there exists a diagonal matrix $D$ with positive diagonal entries such that $DB$ is skew-symmetric. Then, it was shown in \cite{FZ03a} that $B$ may be expressed as
\begin{equation}\label{eq: matrix}
B=\left(\begin{matrix}
0 & -z' & y\\
z & 0 & -x'\\
-y' & x & 0
\end{matrix}\right)
\end{equation}
where $x,y,z,x',y',z' \in \mathbb{R}$ satisfy $xyz=x'y'z'$. In the usual cluster algebra theory, we often suppose that all entries of $B$ are integral, since entries of $B$ appear in the exponent of rational fractions. However, since we do not focus on cluster variables, we may consider real skew-symmetrizable matrices. When we focus on only matrix mutations $\mu_k(B)$, we can naturally extend the definition, and their properties are helpful for the study of cluster algebras. This generalization for real entries have already appeared in \cite{BBH11, FT19, FT23} for the skew-symmetric case.
\par
The rules of mutations depend on the signs of entries in the matrix. If all $x,x',y,y',z$, and $z'$ are either positive or negative, then $B$ is called {\em cyclic}. Otherwise, $B$ is called {\em acyclic}. This word comes from the definitions in the corresponding quivers $Q(B)$. (See Figure~\ref{fig: positive-cyclic} and Figure~\ref{fig: negative-cyclic}.)  If every matrix obtained by applying mutations to $B$ is cyclic, then $B$ is called {\em cluster-cyclic}, and otherwise, we call $B$ {\em cluster-acyclic}. Our purpus is to identify whether $B$ is cluster-cyclic or not.

\subsection{Previous works}
In the skew-symmetric case, this problem was solved in \cite{BBH11,FT19}. In \cite{BBH11}, they showed it for the integer skew-symmetric case, and in \cite{FT19}, they generalized it for real skew-symmetric case by using the geometric aspect of mutations. They showed that the {\em Markov constants}, which was introduced in \cite{BBH11} for the cluster algebra theory, play important roles in this problem. For the skew-symmetrizable case, the Markov constants $C(B)=xx'+yy'+zz'-|xyz|$, which are a generalization of the one for skew-symmetric matrices, also play important roles in this problem. For the skew-symmetrizable case, this problem was solved by Seven \cite{Sev12} by using a {\em quasi-Cartan companion}.
\par
Other motivation of this study is in the classification of the global structures of $G$-fans \cite{Nak24}.
\begin{rem}
After the author submitted this paper, I was informed by Ahmet Seven of the related works in \cite{Sev12,Sev13}. By using results in \cite{Sev12}, Theorem~\ref{thm: 1} can be shown. A related result to Theorem~\ref{thm: fundamental domain} appeared in \cite{Sev13}. However, to obtain other results, we necessitate a different proof that draws upon certain facts from the proof of Theorem~\ref{thm: 1}. So, we give an alternative proof of Theorem~\ref{thm: 1} and Theorem~\ref{thm: fundamental domain} in this paper.
\end{rem}

\subsection{Method}
To study the cluster-cyclicity, we combine the method of \cite{FZ03a} and \cite{BBH11}. In \cite{FZ03a}, there is a map from skew-symmetrizable matrices to skew-symmetric matrices, which we write $\mathrm{Sk}(B)$ in this paper, such that this operation $\mathrm{Sk}$ and mutations $\mu_k$ are commutative. (See Proposition~\ref{prop: fundamental property of DSSS}.) By using this correspondence, some methods for skew-symmetric matrices in \cite{BBH11} work well for skew-symmetrizable matrices. By translating the statements of skew-symmetrizable matrices into the ones of skew-symmetric matrices, some results have already been shown in \cite{FT19}. In particular, Theorem~\ref{thm: 1} below is immediately shown. However, since we use the method of the proof in \cite{BBH11} in Section~\ref{sec: the case (B)} and Section~\ref{sec: integer skew-symmetrizable case}, we give another proof which is similar to the method of \cite{BBH11}.
\par
However, by considering skew-symetrizable matrices, we need some reconstructions. One of the main reasons is that the surjectivity of Markov constants. (See Theorem~\ref{thm: surjectivity}, Example~\ref{ex: examples of 0<C<4}.) In \cite{BBH11}, it was shown that there are no integer skew-symmetric and cluster-cyclic matrices with $0<C(B)<4$. This property simplifies certain arguments compared to the skew-symmetrizable case. However, due to this difference, we require some modifications to its proof.
\subsection{Main results}
First, we mention the identification of cluster-cyclicity. It is a natural generalization of \cite{BBH11} for the skew-symmetrizable and real case. The following result has already obtained in \cite{Sev12} essentially by using a quasi-Cartan companion.
\begin{thm}[cf.~{\cite[Thm.~1.2~(1),(2)]{BBH11}, \cite[Thm.~2.6]{Sev12}}]\label{thm: 1}
Let $B$ be a skew-symmetrizable and cyclic matrix of the form (\ref{eq: matrix}).
Define its {\em Markov constant}
\begin{equation}
C(B)=xx'+yy'+zz'-|xyz|.
\end{equation}
Then, the following two conditions are equivalent.
\begin{itemize}
\item[1.] $B$ is cluster-cyclic. 
\item[2.] It holds that $C(B)\leq 4$ and $xx',yy',zz' \geq 4$.
\end{itemize}
\end{thm}
For some thechnical reasons, it is better to consider transformations $\gamma_k=-\mu_k$ than the usual mutations $\mu_k$ for any $k=1,2,3$. Let $\Gamma$ be a group generated by $\gamma_1$, $\gamma_2$, and $\gamma_3$. When $B$ of the form (\ref{eq: matrix}) satisfies $x,y,z,x',y',z' > 0$, $B$ is called {\em positive-cyclic}. By considering $\gamma_k$, if $B$ is positive and cluster-cyclic, then every element in this orbit $\Gamma(B)$ is also so.
\par
Next result is to give a fundamental domain for this action $\Gamma$ of {\em integer} skew-symmetrizable matrices. The following partial order $\leq$ on skew-symmetrizable and positive-cyclic matrices is useful in this study.
\begin{equation}\label{eq: order of matrices}
\begin{aligned}
&\ \left(\begin{matrix}
0 & -z' & y\\
z & 0 & -x'\\
-y' & x & 0
\end{matrix}\right)
\leq
\left(\begin{matrix}
0 & -w' & v\\
w & 0 & -u'\\
-v' & u & 0
\end{matrix}\right)
\\
\Longleftrightarrow&\ 
x \leq u,\ y \leq v,\ z \leq w,\ x' \leq u',\ y' \leq v',\ z' \leq w'.
\end{aligned}
\end{equation}
By using this order, we can give a fundamental domain for $\Gamma$, which is also a natural generalization of \cite[Thm.~5.1]{BBH11}.
\begin{thm}[Theorem~\ref{thm: fundamental domain}]\label{thm: fundamental domain of skew-symmetrizable}
A set $\tilde{F}$ defined by
\begin{equation}
\left\{\left.
B=\left(\begin{matrix}
0 & -z' & y\\
z & 0 & -x'\\
-y' & x & 0\\
\end{matrix}\right)\ 
\right|
\ \begin{aligned}
&\textup{$B$ is skew-symmetrizable},\\
&xyz \geq 2xx',2yy',2zz',\\
&x,y,z,x',y',z' \in \mathbb{Z}_{>0}
\end{aligned}
\right\}
\end{equation}
is a fundamental domain for $\Gamma$ on integer skew-symmetrizable, positive and cluster-cyclic matrices. Moreover, for any integer skew-symmetrizable, positive, and cluster-cyclic matrix $B$, the element of $\Gamma(B) \cap \tilde{F}$ is the smallest element in $\Gamma(B)$.
\end{thm}
Namely, for each integer skew-symmetrizable, positive, and cluster-cyclic matrix $B$, there is the minimum element of $\Gamma(B)$, and $\tilde{F}$ is the set of all such minimum elements. Unfortunately, this result cannnot generalize for the real case. (See Theorem~\ref{thm: identification of (B)}.)
\par
These two results are natural generalizations of the skew-symmetric case. On the other hand, the following result is the biggest difference between the skew-symmetrizable case and the skew-symmetric case. 
\begin{thm}[Theorem~\ref{thm: surjectivity}]
The Markov constant $C(B)$ can take any integer less than or equal to $4$ by choosing an integer skew-symmetrizable and cluster-cyclic matrix $B$.
\end{thm}
Because of this property (in particular, the fact that there is a cluster-cyclic matrix $B$ with $0 < C(B) < 4$), we need to reconstruct a proof in \cite{BBH11} even though we focus on the integer skew-symmetrizable matrices (see Example~\ref{ex: examples of 0<C<4}).
\subsection{Structure of the paper}
In section~\ref{sec: definition}, we recall the properties of mutations introduced by \cite{FZ02, FZ03a}.
\par
In section~\ref{sec: Markov constant}, we define the Markov constant. This is a generalization for the skew-symmetrizable case of the one which is introduced in \cite{BBH11}.
\par
In section~\ref{section: reduction}, we introduce more simple notations, which is similar to \cite{BBH11, FT19}. In this notations, we may decompose Theorem~\ref{thm: 1} into Lemma~\ref{lem: 2.3}, Lemma~\ref{lem: 2.2}, and Lemma~\ref{lem: 2.1}.
\par
In section~\ref{sec: proof of lemma 2.3}, we show Lemma~\ref{lem: 2.3}.
\par
In section~\ref{sec: 1,2-orbit}, we show that some mutation can be expressed by the Chebyshev polynomials of the second kind. Based on this, we may prove Lemma~\ref{lem: 2.2}. These ideas have appeared in \cite{FT19}.
\par
In section~\ref{sec: Mk conditions} and \ref{sec: proof of lemma 2.1}, we show Lemma~\ref{lem: 2.1}. Most ideas are similar to the one in \cite{BBH11}. However some modifications are necessarily to generalize it for the real and skew-symmetrizable case.
\par
In section~\ref{sec: integer skew-symmetrizable case} and \ref{sec: the case (B)}, we derive some new results. In section~\ref{sec: integer skew-symmetrizable case}, we deal with the integer skew-symmetrizable case. In particular, we focus on the difference between the result of the skew-symmetric case in \cite{BBH11} and the one of the skew-symmetrizable case. In section~\ref{sec: the case (B)}, we consider the case which does not occur when we consider integer skew-symmetric matrces.

\begin{acknowledgement}
This problem was suggested by Tomoki Nakanishi, and I thank him for various important advices. I also thank Zhichao Chen for careful reading and advices.
\par
I thank Ahmet Seven for noticing the important related works.
\par
This work was financially supported by JST SPRING, Grant Number JPMJSP2125. The author would like to take this opportunity to thank the 
“THERS Make New Standards Program for the Next Generation Researchers.” 
\end{acknowledgement}

\section{Definitions}\label{sec: definition}
\subsection{Skew-symmetrizable matrix and mutation}
In this paper, we will focus on skew-symmetrizable matrices of rank $3$ with real entries. Let $\mathrm{M}_3(\mathbb{R})$ be the set of all $3 \times 3$ real matrices. A real matrix $B=(b_{ij}) \in \mathrm{M}_3(\mathbb{R})$ is called {\em skew-symmetrizable} if there exists a diagonal matrix $D=\mathrm{diag}(d_1,d_2,d_3)$ with positive real numbers $d_1,d_2,d_3>0$ such that $DB$ is skew-symmetric. This $D$ is said to be a {\em skew-symmetrizer} of $B$.
\par
Skew-symmetrizable matrices have the following properties.
\begin{prop}[{\cite[Lem.~7.4]{FZ03a}}]\label{prop: fundamental property of skew-symmetrizable}
Let $B=(b_{ij}) \in \mathrm{M}_3(\mathbb{R})$ be a skew-symmetrizable matrix. Then, we have the following statements.
\\
(a)\ $B$ is sign-skew-symmetric, namely, $b_{ij}$ and $-b_{ji}$ have the same sign for any $i,j$. (In particular, every diagonal entry is $0$.)
\\
(b)\ We have
\begin{equation}\label{eq: symmetrizable condition}
b_{12}b_{23}b_{31}=-b_{32}b_{21}b_{13}.
\end{equation}
\end{prop}
Conventionally, the sign $\varepsilon_x$ of a real number $x \in \mathbb{R}$ is defined as follows: $\varepsilon_x=0$ if $x=0$ and, otherwise, $\varepsilon_x=\pm1$ with $\varepsilon_x x > 0$.
Next, we difine mutations of skew-symmetrizable matrices.
\begin{dfn}
Let $B=(b_{i,j}) \in \mathrm{M}_3(\mathbb{R})$ be a skew-symmetrizable matrix. We define {\em mutations} $\mu_{k}(B)=(b'_{ij})\in\mathrm{M}_3(\mathbb{R})$\ ($k=1,2,3$) as follows:
\begin{equation}
b'_{ij}=\begin{cases}
-b_{ij} & i=k\ \textup{or}\ j=k,\\
b_{ij}+b_{ik}[b_{kj}]_{+} + [-b_{ik}]_{+}b_{kj} & i,j \neq k,
\end{cases}
\end{equation}
where $[a]_{+}$ is defined as $[a]_{+}=\max(a,0)$.
\end{dfn}
It is known that $\mu_k\mu_k(B)=B$ holds for any $B \in \mathrm{M}_3(\mathbb{R})$ and $k=1,2,3$ \cite{FZ02}. We can check that every skew-symmetrizer of $B$ is also a skew-symmetrizer of $\mu_k(B)$.
\par
For any skew-symmetrizable matrix $B$, we may define the corresponding quiver $Q(B)$ wighted by $\mathbb{R}_{\geq 0}^2$ as follows:
\begin{itemize}
\item The vertices of $Q(B)$ are indexed by $1$, $2$, and $3$.
\item For each vertex pair $i\neq j$, if $b_{ij}>0$ or, equivalently, $b_{ji}<0$, there is an edge $i\overset{(b_{ij},-b_{ji})}{\longrightarrow}j$. If $b_{ij}<0$ or $b_{ji}>0$, there is an edge $j \overset{(b_{ji},-b_{ij})}{\longrightarrow}i$. If $b_{ij}=b_{ji}=0$, there are no edges between $i$ and $j$.
\end{itemize}
A skew-symmetrizable matrix $B$ is called {\em cyclic} if the corresponding quiver $Q(B)$ is cyclic. Otherwise, $B$ is said to be {\em acyclic}. By definition, $B$ is cyclic if and only if $B$ may be expressed as
\begin{equation}\label{eq: exchange matrix}
\left(\begin{matrix}
0 & -z' & y\\
z & 0 & -x'\\
-y' & x & 0
\end{matrix}\right)
\end{equation}
for some $x,y,z,x',y',z' \in \mathbb{R}$ such that all of these numbers are either positive or negative. When all of these numbers are positive (resp. negative), this matrix is said to be {\em positive-cyclic} (resp. {\em negative-cyclic}). The corresponding quivers are described as in Figure~\ref{fig: positive-cyclic} and Figure~\ref{fig: negative-cyclic}.
\begin{figure}[hbtp]
\centering
\begin{tabular}{cc}
\begin{minipage}{0.5\linewidth}
\centering
\begin{tikzpicture}
\draw (-1.3,0) node {$1$};
\draw (1.3,0) node {$2$};
\draw (0,1.8) node {$3$};
\draw [<-] (-0.2,1.5)--(-0.6,0.9) node[left]{$(y,y')$} --(-1,0.3);
\draw [<-] (1,0.3)--(0.6,0.9)node[right]{$(x,x')$}--(0.2,1.5);
\draw[<-] (-1,0)--(0,0)node[above]{$(z,z')$}--(1,0);
\end{tikzpicture}
\caption{Positive}\label{fig: positive-cyclic}
\end{minipage}
\begin{minipage}{0.5\linewidth}
\centering
\begin{tikzpicture}
\draw (-1.3,0) node {$1$};
\draw (1.3,0) node {$2$};
\draw (0,1.8) node {$3$};
\draw [->] (-0.2,1.5)--(-0.6,0.9) node[left]{$(-y',-y)$} --(-1,0.3);
\draw [->] (1,0.3)--(0.6,0.9)node[right]{$(-x',-x)$}--(0.2,1.5);
\draw[->] (-1,0)--(0,0)node[above]{$(-z',-z)$}--(1,0);
\end{tikzpicture}
\caption{Negative}\label{fig: negative-cyclic}
\end{minipage}
\end{tabular}
\end{figure}
\par
Next, we define the cluster-cyclicity. This is the main contents in this paper.
\begin{dfn}
Let $B \in \mathrm{M}_3(\mathbb{R})$ be a skew-symmetrizable matrix. If every matrix which is obtained by mutating $B$ finitely many times is cyclic, then $B$ is said to be {\em cluster-cyclic}; otherwise, $B$ is said to be {\em cluster-acyclic}.
\end{dfn}
For any skew-symmetrizable matrix $B$, if $B$ is positive-cyclic (resp. negative-cyclic), then $\mu_k(B)$ is not positive-cyclic because every entry in the $k$-th row or column changes their signs. In particular, if $B$ is cluster-cyclic, positive-cyclic and negative-cyclic matrices are generated alternately by applying mutations.

\subsection{Double-sided-skew-symmetrized form}
Our subsequent purpus is to identify whether a given skew-symmetrizable matrix $B \in \mathrm{M}_3(\mathbb{R})$ is cluster-cyclic or not. For this, it is enough to consider the certain skew-symmetric matrix, which we name a {\em double-sided-skew-symmetrized form}. It was introduced by \cite[Lem.~8.3]{FZ03a} when they consider the mutations of diagrams corresponding to skew-symmetrizable matrices.
\par
For any real number $x$, we write the sign of $x$ by $\varepsilon_{x} \in \{0,\pm 1\}$.
\begin{dfn}\label{dfn: double-sided-skew-symmetrized matrix}
Let
\begin{equation}
B=\left(\begin{matrix}
0 & -z' & y\\
z & 0 & -x'\\
-y' & x & 0
\end{matrix}\right) \in \mathrm{M}_3(\mathbb{R})
\end{equation}
be a (not necessarily cyclic) skew-symmetrizable matrix. (Note that $x$, $y$, and $z$ have the same sign as $x'$, $y'$, and $z'$, respectively.) Set $p=\varepsilon_x\sqrt{xx'}$,\ $q=\varepsilon_y\sqrt{yy'}$,\ and $r=\varepsilon_z\sqrt{zz'}$. Then, define the skew-symmetric matrix $\mathrm{Sk}(B) \in \mathrm{M}_3(\mathbb{R})$ as
\begin{equation}
\mathrm{Sk}(B)=\left(\begin{matrix}
0 & -r & q\\
r & 0 & -p\\
-q & p & 0
\end{matrix}\right).
\end{equation}
We call it the {\em double-sided-skew-symmetrized form} of $B$.
\end{dfn}
The name of {\em double-sided-skew-symmetrized} comes from the following property. 
\begin{lem}[{\cite[Lem.~8.3]{FZ03a}}]
Let $B \in \mathrm{M}_3(\mathbb{R})$ be a skew-symmetrizable matrix. Then, for arbitrary skew-symmetrizer $D=\mathrm{diag}(d_1,d_2,d_3)$ of $B$, by choosing its square root $D^{1/2}=\mathrm{diag}(\sqrt{d_1},\sqrt{d_2},\sqrt{d_3})$, it holds that
\begin{equation}\label{eq: double-sided expression}
D^{\frac{1}{2}}BD^{-\frac{1}{2}}=\mathrm{Sk}(B),
\end{equation}
where $D^{-1/2}=(D^{1/2})^{-1} = \mathrm{diag}\left(\sqrt{d_1}^{-1},\sqrt{d_2}^{-1},\sqrt{d_3}^{-1}\right)$.
\end{lem}
In the notations of Definition~\ref{dfn: double-sided-skew-symmetrized matrix}, we often use the following equality.
\begin{equation}\label{eq: pqr lemma}
xyz=x'y'z'=pqr.
\end{equation}
The first equality follows from Proposition~\ref{prop: fundamental property of skew-symmetrizable}~(b). The second equality may be verified as follows:
\begin{equation}
xyz=\varepsilon_{xyz}\sqrt{|xyz|}\sqrt{|x'y'z'|}=\varepsilon_x\varepsilon_y\varepsilon_z\sqrt{|xx'|}\sqrt{|yy'|}\sqrt{|zz'|}=pqr.
\end{equation}
Note that $|xyz|=|x'y'z'|=\sqrt{|xyz|}\sqrt{|x'y'z'|}$.
\par
This double-sided-skew-symmetrized form is compatible for the cluster-cyclicity.
\begin{prop}\label{prop: fundamental property of DSSS}
Let $B \in \mathrm{M}_3(\mathbb{R})$ be a skew-symmetrizable matrix. Then, the following statements hold:\\
(a) Every entry of $\mathrm{Sk}(B)$ has the same sign as the corresponding entry of $B$. In patricular, $B$ is cyclic if and only if $\mathrm{Sk}(B)$ is cyclic.\\
(b)\textup{{\cite[Lem.~8.4]{FZ03a}}} For any $k=1,2,3$, we have
\begin{equation}
\mu_k(\mathrm{Sk}(B))=\mathrm{Sk}(\mu_k(B)).
\end{equation}
\par
In particular, $B$ is cluster-cyclic if and only if $\mathrm{Sk}(B)$ is cluster-cyclic.
\end{prop}
The statement (a) is obvious by the definition, and (b) have been shown in \cite{FZ03a}. By combining (a) and (b), we can show that the last staement by the induction on the number of mutations.

\section{Markov constant}\label{sec: Markov constant}
In this section, we define the Markov constant of a skew-symmetrizable matrix. In \cite{BBH11}, it has an important role in the problem for the integer skew-symmetric case. It is also important for the skew-symmetrizable case.
\par
In this section, we fix a skew-symmetrizable matrix and its double-sided-skew-symmetrized form as
\begin{equation}\label{eq: skew-symmetrizable matrix}
B=\left(\begin{matrix}
0 & -z' & y\\
z & 0 & -x'\\
-y' & x & 0
\end{matrix}\right),
\ 
\mathrm{Sk}(B)=\left(\begin{matrix}
0 & -r & q\\
r & 0 & -p\\
-q & p & 0
\end{matrix}\right).
\end{equation}
as in Definition~\ref{dfn: double-sided-skew-symmetrized matrix}. By Proposition~\ref{prop: fundamental property of skew-symmetrizable} and (\ref{eq: pqr lemma}), these entries satisfy the following conditions:
\begin{itemize}
\item In each triplet $(x,x',p)$, $(y,y',q)$, and $(z,z',r)$, the three entries have the same sign.
\item It holds that $xyz=x'y'z'=pqr$.
\end{itemize}
The following definition is suggested by Tomoki Nakanishi. 
\begin{dfn}\label{dfn: Markov constant}
For any cyclic matrix $B \in \mathrm{M}_3(\mathbb{R})$ of the form (\ref{eq: skew-symmetrizable matrix}), define the {\em Markov constant}
\begin{equation}
C(B)=xx'+yy'+zz'-|xyz|.
\end{equation}
\end{dfn}
The following properties are fundamental.
\begin{prop}\label{prop: expressions of Markov constants}
Let $B$ be a skew-symmetrizable matrix.
\\
(a)\ We have $C(\mathrm{Sk}(B))=C(B)$.\\
(b)\ For any $k=1,2,3$, we have $C(\mu_k(B))=C(B)$.
\end{prop}
\begin{proof}
The statement (a) follows from (\ref{eq: pqr lemma}). We show (b). Set
\begin{equation}
B=\left(\begin{matrix}
0 & -z' & y\\
z & 0 & -x'\\
-y' & x & 0
\end{matrix}\right),
\ 
\mu_k(B)=\left(\begin{matrix}
0 & w' & -v\\
-w & 0 & u'\\
v' & -u & 0
\end{matrix}\right).
\end{equation}
Suppose that $B$ is positive-cyclic. Regardless $B$ and $\mu_k(B)$ is cyclic or not, we have $\varepsilon_{\mu_k(B)}|uvw|=uvw$. Thus, we have
\begin{equation}
C(\mu_k(B))=uu'+vv'+ww'-uvw.
\end{equation}
We can check $C(\mu_k(B))=C(B)$ by a direct calculation. For example, if $k=1$, we have
\begin{equation}\label{eq: invariancy of the Markov constant}
\begin{aligned}
C(\mu_1(B))&=(y'z'-x)(yz-x')+yy'+zz'-(y'z'-x)yz\\
&=xx'+yy'+zz'-x'y'z'=C(B).
\end{aligned}
\end{equation}
Suppose that $B$ is acyclic. For the simplicity, we assume $x,x' \leq 0$ and $y,y',z,z' > 0$. Then, regardless of the cyclicity of $\mu_k(B)$, we have $\varepsilon_{\mu_k(B)}|uvw|=-uvw$ holds. By a similar calculation of (\ref{eq: invariancy of the Markov constant}), we may show $C(\mu_k(B))=C(B)$.
\end{proof}

\section{Notations}\label{section: reduction}
In this section, we give more simple notations as in \cite{BBH11}, and we translate Theorem~\ref{thm: 1} into these notations.
\par
Our purpose is to identify whether a given skew-symmetrizable matrix $B$ is cluster-cyclic or not. In this case, $\gamma_k=-\mu_k$ is more convenient than the usual mutation $\mu_k$. Note that $-\mu_k(B)=\mu_k(-B)$ for any skew-symmetrizable matrix $B$. Thus, $B$ is cluster-cyclic if and only if every $\gamma_{i_s}\cdots\gamma_{i_1}(B)=(-1)^{s}\mu_{i_s}\cdots\mu_{i_1}(B)$ is cyclic.
\par
Moreover, for any skew-symmetrizable matrix
\begin{equation}
B=\left(\begin{matrix}
0 & -z' & y\\
z & 0 & -x'\\
-y' & x & 0
\end{matrix}\right),
\end{equation}
we will write it as the following tuplet:
\begin{equation}
M=\left(\begin{matrix}
x & y & z\\
x' & y' & z'
\end{matrix}\right).
\end{equation}
By Proposition~\ref{prop: fundamental property of skew-symmetrizable}, these numbers satisfy the following properties.
\begin{equation}\label{eq: properties of M}
\begin{aligned}
&\textup{The equality $xyz=x'y'z'$ holds}\\
&\textup{and the sign of two entries on the same row is the same.}
\end{aligned}
\end{equation}
Moreover, $B$ is cyclic if and only if all entries of $M$ have the same sign.
If $B$ is cyclic, the maps $\gamma_k$ may be expressed as
\begin{equation}\label{eq: gamma on M}
\begin{aligned}
\gamma_1(M)&=\left(\begin{matrix}
y'z'-x & y & z\\
yz-x' & y' & z'
\end{matrix}\right),\\
\gamma_2(M)&=\left(\begin{matrix}
x & z'x'-y & z\\
x' & zx-y' & z'
\end{matrix}\right),\\
\gamma_3(M)&=\left(\begin{matrix}
x & y & x'y'-z\\
x' & y' & xy-z'
\end{matrix}\right).
\end{aligned}
\end{equation}
Moreover, by Proposition~\ref{prop: fundamental property of DSSS}~(c), when we consider whether a skew-symmetrizable matrix $B \in \mathrm{M}_3(\mathbb{R})$ is cluster-cyclic or not, it suffices to consider whether
\begin{equation}
\mathrm{Sk}(B)=\left(\begin{matrix}
0 & -r & q\\
r & 0 & -p\\
-q & p & 0
\end{matrix}\right)
\end{equation}
is cluster-cyclic or not. We may see this matrix as a triplet of real numbers $(p,q,r) \in \mathbb{R}^3$. If $\mathrm{Sk}(B)$ is cyclic, the maps $\gamma_k$ may be expressed as
\begin{equation}\label{eq: gamma on S}
\begin{aligned}
\gamma_1(p,q,r)&=(qr-p,q,r),\\
\gamma_2(p,q,r)&=(p,rp-q,r),\\
\gamma_3(p,q,r)&=(p,q,pq-r).
\end{aligned}
\end{equation}
Our interest is for cyclic matrices, and we do not treat acyclic matrixes. So, we see (\ref{eq: gamma on M}) and (\ref{eq: gamma on S}) as the definition of $\gamma_k$.
\par
We summerize the definitions.
\begin{dfn}\label{dfn: definition of sets}
Define sets
\begin{equation}
\begin{aligned}
&\mathcal{M}&&=\left\{\left. M=\left(\begin{matrix}
x & y & z\\
x' & y' & z'
\end{matrix}\right)\ \right|\ \textup{$M$ satisfies (\ref{eq: properties of M})}\ \right\},\\
&\mathcal{M}^{+}&&=\left\{\left.\left(\begin{matrix}
x & y & z\\
x' & y' & z'
\end{matrix}\right) \in \mathcal{M}\ \right|\ x,y,z,x',y',z' > 0\right\},\\
&\mathcal{S}&&=\{(p,q,r)\mid p,q,r \in \mathbb{R}\}=\mathbb{R}^3,\\
&\mathcal{S}^{+}&&=\{(p,q,r) \in \mathcal{S}\mid p,q,r > 0\}.
\end{aligned}
\end{equation}
Every element of $\mathcal{M}^{+}$ and $\mathcal{S}^{+}$ is called {\em positive}. We define maps $\gamma_k :\mathcal{M} \to \mathcal{M}$ and $\gamma_k:\mathcal{S} \to \mathcal{S}$ as (\ref{eq: gamma on M}) and (\ref{eq: gamma on S}), respectively.
\end{dfn}
\begin{dfn}
For any finite sequence $t=(t_1,\dots,t_k)$ of indices $1,2,3$, set $\gamma_t=\gamma_{t_k}\cdots\gamma_{t_1}$. For any $M \in \mathcal{M}$, let $\Gamma(M)$ be the collection of $\gamma_t(M)$ indexed by all of finite sequences $t=(t_1,\dots,t_k)$ of indices $1,2,3$ with $t_i \neq t_{i+1}$, where $\gamma_t=\gamma_{t_k}\cdots\gamma_{t_1}$. For any positive element $M \in \mathcal{M}^{+}$, if every element in $\Gamma(M)$ is positive, $M$ is said to be {\em cluster-positive}. Similarly, we define the cluster-positivity of $S \in \mathcal{S}^{+}$.
\par
Here, we define a map $\mathrm{Sk}:\mathcal{M} \to \mathcal{S}$ as
\begin{equation}
\mathrm{Sk}\left(\begin{matrix}
x & y & z\\
x' & y' & z'
\end{matrix}\right)
=(\varepsilon_x\sqrt{xx'},\varepsilon_y\sqrt{yy'},\varepsilon_z\sqrt{zz'}),
\end{equation}
where $\varepsilon_a \in \{0,\pm1\}$ is the sign of $a \in \mathbb{R}$. If $M \in \mathcal{M}$ corresponds to a positive-cyclic matrix $B \in \mathrm{M}_3(\mathbb{R})$, then $\mathrm{Sk}(M)$ corresponds to its double-sided skew-symmetrized form $\mathrm{Sk}(B)$. Moreover, $M$ is cluster-positive if and only if $S=\mathrm{Sk}(M)$ is cluster-positive.
\end{dfn}
\begin{dfn}
For any $M=\left(\begin{smallmatrix}
x & y & z\\
x' & y' & z'
\end{smallmatrix}\right) \in \mathcal{M}$ and $S=(p,q,r) \in \mathcal{S}$, we define the Markov constants $C(M)=xx'+yy'+zz'-xyz$ and $C(S)=p^2+q^2+r^2-pqr$.
\end{dfn}
By definition, we have $C(\mathrm{Sk}(M))=C(M)$.
Note that, for any positive-cyclic and skew-symmetrizable matrix $B$, it corresponds to an element $M \in \mathcal{M}^{+}$. Then, it holds that
\begin{equation}\label{eq: relationship of C}
C(M)=C(B). 
\end{equation}
Moreover, we may check
\begin{equation}\label{eq: Markov constant of S}
C(\gamma_k(M))=C(M).
\end{equation}
\par
We summarize the results in this section.
\begin{prop}\label{prop: relationship between B and S}
Let
\begin{equation}
B=\left(\begin{matrix}
0 & -z' & y\\
z & 0 & -x'\\
-y' & x & 0
\end{matrix}\right)
\in \mathrm{M}_3(\mathbb{R})
\end{equation}
be skew-symmetrizable and positive-cyclic. Set $p=\sqrt{xx'}$, $q=\sqrt{yy'}$, $r=\sqrt{zz'}$, and $S=(p,q,r)\in \mathcal{S}^{+}$.
\\
(a)\textup{(Proposition~\ref{prop: fundamental property of DSSS})}\ $B$ is cluster-cyclic if and only if $S$ is cluster-positive.
\\
(b)\textup{(\ref{eq: relationship of C})}\ We have $C(B)=C(S)$.
\\
(c)\textup{(\ref{eq: Markov constant of S})} On $\mathcal{S}$, the Markov constant does not change under the transformations of $\gamma_k$.
\end{prop}
Based on these results, we may translate Theorem~\ref{thm: 1} into these notations. This is the generalization of {\cite[Thm.~1.2~(1),(2)]{BBH11}}, and it was shown in \cite{FT19}.
\begin{prop}[{\cite[Lem.~3.3, Thm.~4.5, Cor.~4.8]{FT19}}]\label{prop: 1}
Let $S=(p,q,r) \in \mathcal{S}^{+}$. Then, $S$ is cluster-positive if and only if both $p,q,r \geq 2$ and $C(S) \leq 4$ hold.
\end{prop}
\begin{rem}
This result has essentially generalized for skew-symmetrizable case by \cite{Sev12}. In this paper, the result is expressed by the {\em quasi-Cartan companion} instead of the Markov constant.
\end{rem}
Now, we show Theorem~\ref{thm: 1} under the assumption of Proposition~\ref{prop: 1}.
\begin{proof}[Proof of Theorem~\ref{thm: 1}]
Let $B$ be a skew-symmetrizable and cyclic matrix. If $B$ is positive-cyclic, the claim holds by using the correspondence of $\mathrm{Sk}(B)$ and $S$. When $B$ is negative-cyclic, then $-B$ is positive-cyclic. Thus, by applying the result of Proposition~\ref{prop: 1} for $-B$, we can show Theorem~\ref{thm: 1}.
\end{proof}
Proposition~\ref{prop: 1} was shown in \cite{FT19} by using geometrical aspects of mutations. (Thus, we complete the proof of Theorem~\ref{thm: 1}.) However, we give an alternatively proof of this proposition, which is more closely related to \cite{BBH11}.
\par
We decompose Proposition~\ref{prop: 1} into the following three statemenets.
\begin{lem}\label{lem: 2.3}
Let $S=(p,q,r) \in \mathcal{S}^{+}$ satisfy $C(S) \leq 4$ and $p,q,r \geq 2$. Then, $S$ is cluster-positive.
\end{lem}
\begin{lem}\label{lem: 2.2}
Let $S = (p,q,r) \in \mathcal{S}^{+}$ be cluster-positive. Then, we have $p,q,r \geq 2$.
\end{lem}
\begin{lem}\label{lem: 2.1}
Let $S \in \mathcal{S}^{+}$ be cluster-positive. Then, we have $C(S) \leq 4$.
\end{lem}

\section{Proof of Lemma~\ref{lem: 2.3}} \label{sec: proof of lemma 2.3}
Here, we show Lemma~\ref{lem: 2.3}. The essential part is the following.
\begin{lem}\label{lem: p,q,r >=2}
Let $S=(p,q,r) \in \mathcal{S}^{+}$ satisfy $q,r \geq 2$ and $C(S) \leq 4$. Then, we have $p \geq 2$.
\end{lem}
\begin{proof}
We have
\begin{equation}
C(S)=p^2-pqr+q^2+r^2=\left(p-\frac{1}{2}qr\right)^2-\frac{1}{4}q^2r^2+q^2+r^2.
\end{equation}
Thus, the inequality $C(S) \leq 4$ turns into
\begin{equation}
\left(p-\frac{1}{2}qr\right)^2 \leq \frac{1}{4}q^2r^2-q^2-r^2+4=\frac{1}{4}(q^2-4)(r^2-4).
\end{equation}
Since the last term is non-negative, this inequality implies
\begin{equation}
p-\frac{1}{2}qr \geq -\frac{1}{2}\sqrt{(q^2-4)(r^2-4)}.
\end{equation}
In particular, $f(q,r)=\frac{1}{2}(qr-\sqrt{(q^2-4)(r^2-4)})$ is a lower bound of $p$. Thus, it suffices to show that $f(q,r) \geq 2$ for any $q,r \geq 2$. The inequality $f(q,r) \geq 2$ is equivalent to
\begin{equation}
qr-4 \geq \sqrt{(q^2-4)(r^2-4)}.
\end{equation}
Since both sides are non-negative, it is equivalent to $(qr-4)^2 \geq (q^2-4)(r^2-4)$. Then, we have
\begin{equation}
(qr-4)^2-(q^2-4)(r^2-4)=4q^2+4r^2-8qr=4(q-r)^2 \geq 0.
\end{equation}
It implies $f(q,r) \geq 2$ as we said above. This completes the proof.
\end{proof}
By using this lemma, we may show Lemma~\ref{lem: 2.3} inductively.
\begin{proof}[Proof of Lemma~\ref{lem: 2.3}]
Let $S=(p,q,r) \in \mathcal{S}^{+}$ satisfy $p,q,r \geq 2$ and $C(S) \leq 4$ and set $S'=(p',q',r')=\gamma_k(S)$. By definition of $\gamma_k$, all entries other than $k$th one do not change. In particular, they are larger or equal to $2$. By Proposition~\ref{prop: relationship between B and S}~(c), $C(S')=C(S) \leq 4$ holds. Thus, the $k$th entry of $S'$ is also larger or equal to $2$ by Lemma~\ref{lem: p,q,r >=2}. By repeating this argument, every entry of elements in $\Gamma(S)$ is larger or equal to $2$ and, in particular, positive. Thus, the claim holds. 
\end{proof}

\section{$1,2$-orbits and proof of Lemma~\ref{lem: 2.2}} \label{sec: 1,2-orbit}
The contents in this section have already appeared in \cite[Lem.~3.3]{FT19}.
\begin{dfn}
Let $r$ be an indeterminate. Then, define polynomials $u_n(r)$\ ($n \in \mathbb{Z}_{\geq -2}$) recursively as follows:
\begin{equation}
\begin{aligned}
u_{-2}(r)&=-1,\ u_{-1}(r)=0,\\
u_{n+1}(r)&=ru_{n}(r)-u_{n-1}(r)\quad(n \geq -1).
\end{aligned}
\end{equation}
\end{dfn}
Note that $u_{0}(r)=1$ and $u_1(r)=r$. These polynomials are closely related to the Chebyshev polynomials $U_n(r)$ of the second kind, which are obtained as the following recursions:
\begin{equation}
\begin{aligned}
U_{-2}(r)&=-1,\ U_{-1}(r)=0,\\
U_{n+1}(r)&=2rU_{n}(r)-U_{n-1}(r)\quad(n \geq -1).
\end{aligned}
\end{equation}
Note that $U_{0}(r)=1$ and $U_1(r)=2r$. In particular, we have
\begin{equation}\label{eq: relation between u and U}
u_n(r)=U_n\left(\frac{r}{2}\right).
\end{equation}
For any triplet $S=(p,q,r) \in \mathcal{S}$ and integer $n \geq -1$, set $f_n(S)=qu_{n}(r)-pu_{n-1}(r)$. Then, we have $f_{-1}(S)=p$ and $f_0(S)=q$. Moreover, $f_n(S)$ obeys the following recursion:
\begin{equation}
f_{n+1}(S)=rf_n(S)-f_{n-1}(S).
\end{equation}
\begin{dfn}
Let $S \in \mathcal{S}^{+}$. We define a collection $\Gamma_{1,2}(S)=\{\gamma_t(S)\}_{t}$ indexed by a sequence $t=(t_1,\dots,t_l)$ of $1,2$ such that $t_i \neq t_{i+1}$ for any $i=1,\dots, l-1$.
\end{dfn}
Every element of $\Gamma_{1,2}(S)$ may be expressed as follows:
\begin{prop}[{\cite[Lem.~3.3]{FT19}}]\label{prop: fn expression}
Fix $S=(p,q,r)$, and let $p_n$ and $q_n$\ ($n \geq 0$) be
\begin{equation}
(p_n,q_n,r)=\begin{cases}
(\gamma_2\gamma_1)^{\frac{n}{2}}(S) & \textup{$n$ is even},\\
\gamma_1(\gamma_2\gamma_1)^{\frac{n-1}{2}}(S) & \textup{$n$ is odd}.
\end{cases}
\end{equation}
Then, for any $n \in \mathbb{Z}_{\geq 0}$, we have
\begin{equation}\label{eq: pn,qn}
p_n=\begin{cases}
f_{n-1}(S) & \textup{$n$ is even},\\
f_n(S) & \textup{$n$ is odd},
\end{cases}
\quad
q_n=\begin{cases}
f_{n}(S) & \textup{$n$ is even},\\
f_{n-1}(S) & \textup{$n$ is odd}.
\end{cases}
\end{equation}
\end{prop}
We may check it by the induction on $n$. By using this property, entries of $\Gamma_{1,2}(S)$ are closely related to the Chebyshev polynomials.
The following properties are well known.
\begin{prop}[{e.g.~\cite[(1.4), (1.33b)]{HM03}}]\label{prop: chebyshev polynomial}
For any $n \geq 0$ and $\theta \in \mathbb{R}$, we have $\sin (n+1)\theta=\sin \theta U_n(\cos \theta)$ and $\sinh (n+1)\theta = \sinh \theta U_n(\cosh \theta)$. Moreover, we have $U_n(1)=n+1$.
\end{prop}
We give some fundamental properties of $f_n(S)$.
\begin{prop}[{\cite[Lem.~3.3]{FT19}}]\label{prop: r<2}
Let $S=(p,q,r) \in \mathcal{S}^{+}$. If $r<2$, then there is a negative entry in $\Gamma_{1,2}(S)$.
\end{prop}
\begin{proof}
For the reader's convenience, we give a proof following \cite{FT19}. When $0 < r < 2$, we may express $r= 2\cos \theta$ for some $\theta \in \mathbb{R}$ with $0 < \theta <\frac{\pi}{2}$. By Proposition~\ref{prop: fn expression}, $f_n(S)=qu_n(r)-pu_{n-1}(r)$ is a entry of an element in $\Gamma(S)$. We show that there is $n \in \mathbb{Z}_{\geq 0}$ such that $f_n(S)<0$. Note that $u_n(r)=U_n(\cos \theta)=\frac{\sin (n+1)\theta}{\sin \theta}$ by (\ref{eq: relation between u and U}) and Proposition~\ref{prop: chebyshev polynomial}. By using this relation, the inequality $f_n(S)<0$ is equivalent to $q\sin (n+1)\theta < p\sin n\theta$.
Here we use $\sin \theta >0$ since $0<\theta < \frac{\pi}{2}$. We can find $n \geq 1$ such that $\sin (n+1)\theta \leq 0$ while every $k=1,\dots,n$ satifies $\sin k\theta > 0$. This $n$ satisfies $q\sin (n+1)\theta < p\sin n\theta$.
\end{proof}
By using this property, Lemma~\ref{lem: 2.2} is immediately shown as follows.
\begin{proof}[Proof of Lemma~\ref{lem: 2.2}]
More precisely, we show the contraposition of Lemma~\ref{lem: 2.2}. Suppose that $S=(p,q,r) \in \mathcal{S}^{+}$ satisfies $r < 2$. Then, by Proposition~\ref{prop: r<2}, there is a negative entry in $\Gamma_{1,2}(S) \subset \Gamma(S)$. Thus, $S$ is not cluster-positive.
\end{proof} 
For later, we show the following proposition.
\begin{prop}\label{prop: limit theorem}
Let $S=(p,q,r) \in \mathcal{S}^{+}$ with $p,q,r \geq 2$ and set $r=2\cosh \theta$ for some $\theta \in \mathbb{R}_{\geq 0}$. Let $p_n$ and $q_n$ be in (\ref{eq: pn,qn}). If $p_n$ or $q_n$ converges to some real numbers, then we have $qe^{\theta}-p=0$. Moreover,
\\
(a)\ if $r=2$, then we have $p_n=q_n=p=q$ for any $n = 0,1,2,\dots$.
\\
(b)\ if $r>2$, then both $p_n$ and $q_n$ converge to $0$.
\end{prop}
\begin{proof}
Suppose that $p_n$ converges. This is equivalent that $f_{2k+1}$\ ($k \in \mathbb{Z}_{>0}$) converges. Let $n=2k+1$. When $r=2$, it holds that $\theta=0$. Since $u_n(2)=U_n(1)=n+1$, we have $f_{n}(S)=q(n+1)-pn=(q-p)n+q$. Thus, if $f_{n}(S)$ converges, we have $q-p=qe^{0}-p=0$. By considering the maps $\gamma_1$ and $\gamma_2$, we may obtain $p_n=q_n=p=q$ for any $n \in \mathbb{Z}_{>0}$.
If $r > 2$, or equivalently, if $\theta > 0$, it holds that $\sinh n\theta >0$ for any $n \in \mathbb{Z}_{>0}$ and we have
$\frac{\sinh(n+1)\theta}{\sinh n\theta}=\frac{e^{(n+1)\theta}-e^{-(n+1)\theta}}{e^{n\theta}-e^{-n\theta}}=\frac{e^{\theta}-e^{-(2n+1)\theta}}{1-e^{-2n\theta}} \to e^{\theta}$ when $n \to \infty$. Since $u_n(r)=U_n(\cosh \theta)=\frac{\sinh (n+1)\theta}{\sinh n\theta}$, we have
\begin{equation}
\begin{aligned}
f_{n}(S)&=q\frac{\sinh(n+1)\theta}{\sinh \theta}-p\frac{\sinh n\theta}{\sinh \theta}\\
&=\frac{\sinh n\theta}{\sinh \theta}\left(q\frac{\sinh (n+1)\theta}{\sinh n\theta}-p\right).
\end{aligned}
\end{equation}
If $qe^\theta -p \neq 0$, it does not converge since $\sinh n\theta  \to \infty$ when $n \to \infty$. When $p=qe^{\theta}$, we have
\begin{equation}
\begin{aligned}
f_{n}(S)&=q\frac{\sinh(n+1)\theta}{\sinh \theta}-qe^{\theta}\frac{\sinh n\theta}{\sinh \theta}\\
&=\frac{q}{\sinh \theta}\left(\sinh (n+1)\theta - e^{\theta}\sinh\theta\right)\\
&=\frac{q}{2\sinh \theta}\left(-e^{-(n+1)\theta}+e^{(-n+1)\theta}\right)\to 0.
\end{aligned}
\end{equation}
\end{proof}

\section{(M$k$) conditions} \label{sec: Mk conditions}
Later, we will show $C(S) \leq 4$ for any cluster-positive element $S$. It was shown in \cite{FT19} by using the geometrical aspects of mutations. Here, we give an alternative proof which is more closer to that in \cite{BBH11}, althoug we need some modifications for the real case.
\par
In the previous section, for any $S \in \mathcal{S}^{+}$, $\Gamma(S)=\{\gamma_{t}(S)\}_{t}$ was defined as a collection indexed by $t$. However, in this section, we view $\Gamma(S)=\{\gamma_t(S) \mid t\}$ as a set.
\par
Let $\leq$ be a partial order on $\mathcal{S}$ such that every corresponding entry satisfies this inequality on $\mathbb{R}$.
\begin{dfn}\label{dfn: Mk conditions}
For any $S=(p_1,p_2,p_3) \in \mathcal{S}$ and $k=1,2,3$, we say that $S$ satisfies the condition (M$k$) if $S$ satisfies the following:
\begin{itemize}
\item[(M1)] For any $i=1,2,3$, it holds that $S \leq \gamma_i(S)$.
\item[(M2)] For presicely two indices $i$ in $\{1,2,3\}$, it holds that $S \leq \gamma_i(S)$.
\item[(M3)] Otherwise, namely, at most one index $i$ in \{1,2,3\} satisfies $S \leq \gamma_{i}(S)$.
\end{itemize}
\end{dfn}
Acyclic matrices are closely related to (M3) as follows.
\begin{lem}\label{lem: negative lemma}
Let $S=(p,q,r) \in \mathcal{S}$. If $p,q >0$ and $r \leq 0$, then $S$ satisfies (M3).
\end{lem}
\begin{proof}
We have
\begin{equation}
\gamma_1(S)=(qr-p,q,r),\  \gamma_2(S)=(p,pr-q,r).
\end{equation}
Then, since $qr-p<0<p$ and $pr-q<0<q$, this $S$ satisfies $S \nleq \gamma_1(S)$ and $S \nleq \gamma_2(S)$. It implies that $S$ satisfies (M3).
\end{proof}

Under the assumption of an order among $p$, $q$, and $r$, we may identify these conditions (M$k$) as follows.

\begin{prop}\label{prop: charac. of M}
Let $S=(p,q,r) \in \mathcal{S}^{+}$ satisfy $p \geq q \geq r$.
\\
(1)\ A triplet $S$ satisfies (M1) if and only if $S$ satisfies $S \leq \gamma_1(S)$, or equivalently, $qr \geq 2p$.\\
(2)\ A triplet $S$ satisfies (M2) if and only if $2p^2>pqr\geq 2q^2$. In this case, $S$ satisfies $S \nleq \gamma_1(S)$ and $S \leq \gamma_2(S),\gamma_3(S)$.\\
(3)\ A triplet $S$ satisfies (M3) if and only if $2q > pr$. In this case, $S$ satisfies $S \nleq \gamma_1(S),\gamma_2(S)$.
\end{prop}
\begin{proof}
Note that
\begin{quote}
$S \leq \gamma_1(S)$ $\Leftrightarrow$ $qr-p\geq p$ $\Leftrightarrow$ $qr \geq 2p$ $\Leftrightarrow$ $pqr \geq 2p^2$,\\
$S \leq \gamma_2(S)$ $\Leftrightarrow$ $pqr \geq 2q^2$,\\
$S \leq \gamma_3(S)$ $\Leftrightarrow$ $pqr \geq 2r^2$.
\end{quote}
Under the assumption $p \geq q \geq r > 0$, whether $S$ satisfies (M1), (M2), or (M3) is determined by where $pqr$ is in $2p^2 \geq 2q^2 \geq 2r^2$. By considering each condition, we may show Proposition~\ref{prop: charac. of M}
\end{proof}

Following a similar argument as in \cite[Lem.~2.1]{BBH11}, we will prove some fundamental properties of these classes in $\mathcal{S}^{+}$.
\par
First, we show that the condition $(M3)$ does not appear in $\Gamma(S)$ for any cluster-positive $S$.
\begin{lem}[see~{\cite[Lem.~2.1.~(c)]{BBH11}}] \label{lem: M3}
If $S=(p,q,r) \in \mathcal{S}^{+}$ satisfies (M3), then we have $\min(p,q,r) < 2$ and $S$ is not cluster-positive.
\end{lem}
\begin{proof}
Without loss of generality, we may assume $p \geq q \geq r>0$. By Proposition~\ref{prop: charac. of M}~(3), we have $2q>pr$ and $2p > qr$. By multiplying the corresponding sides, we have
\begin{equation}
4pq>pqr^2
\end{equation}
and it implies $4>r^2$, namely, $r<2$. By Proposition~\ref{prop: r<2}, we may verify that $S$ is not cluster-positive.
\end{proof}
By the above lemma, if $S$ is cluster-positive, then $\Gamma(S)$ consists of elements which satisfy either (M1) or (M2). In particular, (M1) gives some strong properties as follows.
\begin{lem}[see~{\cite[Lem~2.1~(a)]{BBH11}}]\label{lem: M1}
Let $S=(p,q,r) \in \mathcal{S}^{+}$ satisfy (M1). Then, the following statements hold.\\
(a)\ We have three inequalities $p,q,r \geq 2$.\\
(b)\ Every element in $\Gamma(S)$ satisfies (M2) except for $S$. Here, we see $\Gamma(S)=\{\gamma_t(S)\mid t\}$ as a set.\\
(c)\ For any sequence
\begin{equation}\label{eq: increasing sequence}
S=S_0 \overset{\gamma_{t_1}}{\mapsto} S_1 \overset{\gamma_{t_2}}{\mapsto} S_2 \overset{\gamma_{t_3}}{\mapsto} \cdots \overset{\gamma_{t_{k}}}{\mapsto} S_k=S'
\end{equation}
which satisfies $t_{i-1} \neq t_{i}$ ($i=1,2,\dots,k$), we have $S_{i-1} \leq S_{i}$. In particular, for any $S' \in \Gamma(S)$, there is a sequence (\ref{eq: increasing sequence}) with $S_{i-1} \leq S_{i}$, and $S_{i-1} \neq S_{i}$ for any $i=1,2,\dots,k$.\\
(d)\ This $S$ is the minimum element in $\Gamma(S)$ in terms of the partial order $\leq$. In particular, $S$ is cluster-positive.
\end{lem}

\begin{proof}
First, we show (a). By Proposition~\ref{prop: charac. of M}~(1), we have
\begin{equation}
pqr \geq 2p^2,2q^2,2r^2
\end{equation}
For example, by multiplying both sides of two inequalities $pqr \geq 2p^2$ and $pqr \geq 2q^2$, we have $(pqr)^2\geq 4p^2q^2$, and it implies $r^2 \geq 4$ and $r \geq 2$. Similarly, we have $p,q \geq 2$.
\par
We show (b) and (c) by the induction on the distance $d$ from $S$ to $S'$. When $d=0$, namely, $S=S'$, it is obvious. Suppose that the claim holds for some $d$ and let the distance form $S$ to $S'$ be $d+1$. Then, we can find $t_{d+1}=1,2,3$ and $S_d=\gamma_{t_{d+1}}(S')$ such that the distance from $S$ to $S_d$ is $d$. By the assumption, we obtain a sequence $S,S_1,\dots,S_d$ satisfying (c) and $S_{i}$ satisfies (M2) for any $i=1,2,\dots,d$. Let $d \geq 1$. Since $S_{d-1}\leq S_{d}= \gamma_{t_{d}}(S_{d-1})$ and $S_{d-1} \neq S_{d}$, we have $S_d \nleq S_{d-1}=\gamma_{t_{d}}(S_d)$. Thus, since $S_d$ satisfies (M2) and $t_{d} \neq t_{d+1}$, we have $S_{d} \leq S_{d+1}$. This inequality also holds for $d=0$ since $S=S_0$ satisfies (M1). If $S_d = S'$, the sequence $S,S_1,\dots,S_d=S'$ satisfies (c). If $S_d \neq S'$, then the sequence $S,S_1,\dots,S_d,S'$ satisfies (c). Thus, (c) holds. By considering this sequence, $S'$ does not satisfy (M1) since $S' \nleq \gamma_{t_{d+1}}(S')$. Moreover, $S'$ does not satisfy (M3). This is because, if $S'=(p',q',r')$ satisfies (M3), then we have $\min(p',q',r') < 2$ by Lemma~\ref{lem: M3}, but it contradicts $(2,2,2) \leq S \leq S'$ by (a) and (c).
\par
Last, we show (d). By (c), we have $S \leq S'$ for any $S' \in \Gamma(S)$. Thus, $S$ is the smallest element of $\Gamma(S)$, and it is cluster-positive. 
\end{proof}

By combining these results, we can give two classes of cluster-positive elements.
\begin{dfn}\label{dfn: A, B conditions}
For any cluster-positive element $S \in \mathcal{S}^{+}$, we define the following conditions.
\begin{itemize}
\item[(A)] $\Gamma(S)$ has a unique element satisfying (M1).
\item[(B)] All elements of $\Gamma(S)$ satisfy (M2).
\end{itemize}
\end{dfn}
By Lemma~\ref{lem: M3} and Lemma~\ref{lem: M1}, every cluster-positive element satisfies either (A) or (B).
For the case of (B), there is a sequence
\begin{equation}\label{eq: (M2) sequence}
S_0=S, S_1=\gamma_{t_1}(S_0), S_2=\gamma_{t_2}(S_1),\dots\quad(t_i \neq t_{i+1}\ \textup{for any $i$})
\end{equation}
such that $S_i$ satisfies (M2) and $S_{i-1} \nleq S_{i}$ for any $i$. Moreover, this sequence is uniquely determined by $S \in \mathcal{S}^{+}$.
\par
Now we identify $\mathcal{S}^{+}$ as $\mathbb{R}^3_{>0}$, and we introduce a topology on $\mathcal{S}^{+}$ by this correspondence. Then, the limit of this sequence also satisfies (M1).
\begin{lem}\label{lem: convergence}
Let $S \in \mathcal{S}^{+}$ satisfy the condition (B) and let $\{S_i\}$ be a sequence of (\ref{eq: (M2) sequence}). Then, we have the following statements.
\\
(a)\ The sequence $\{S_i\}$ converges to $(2,2,2)$.
\\
(b)\ We have $C(S)=4$.
\end{lem}
The existence of the limit was shown in \cite[Thm.~5.1]{BBH11}, but we can show that this limit is $(2,2,2)$.
\begin{proof}
(a)\ First, we show that $S_i$ converges to some $L(S)=(a,b,c)$ with $a,b,c \geq 2$. Let $S_i=(p_i,q_i,r_i)$. Then, since $S_i$ is cluster-positive, we have $p_i,q_i,r_i \geq 2$. Moreover, these sequences $\{p_i\}$, $\{q_i\}$, and $\{r_i\}$ are monotonically decreasing. Thus, they converges to some real numbers $a$, $b$, and $c$, respectively. Thus, $S_i$ converges to $L(S)=(a,b,c)$. Moreover, since $p_i,q_i,r_i \geq 2$, we have $a,b,c \geq 2$. Next, we show $a=b=c=2$. Without loss of generality, we may assume $a \geq b \geq c$. Since monotonically decreasing sequences $p_i$, $q_i$, and $r_i$ converges to $a$, $b$, and $c$, respectively, for any $\varepsilon >0$, there is $N \in \mathbb{Z}_{>0}$ such that $a \leq p_i < a+\varepsilon$, $b \leq q_i < b+\varepsilon$, and $c \leq r_i < c+\varepsilon$ for any $i \geq N$. Since $S_i$ satisfies (M2), there is $i \geq N$ such that $p_i>q_i\geq r_i$. By Proposition~\ref{prop: charac. of M}~(2), we have $q_ir_i < 2p_i$. Since $p_i<a+\varepsilon$ and $bc \leq q_ir_i$, we have $bc < 2(a+\varepsilon)$. This inequality holds for arbtrary $\varepsilon>0$. By considering the limit $\varepsilon \to +0$, we have $bc \leq 2a$. Next, consider $S_{i+1}$. Since $p_i>q_i\geq r_i$ and $S_i \nleq S_{i+1}$, we need to choose $S_{i+1}=\gamma_1(S_i)$. In order $S_{i+1}$ to satisfy (M2), we need $q_{i+1}>p_{i+1},r_{i+1}$. Thus, by doing a similar argument for $S_i$, we have $ac \leq 2b$. By the two inequalities $bc \leq 2a$ and $ac \leq 2b$, we have $c^2\leq 4$ and, by $c \geq 2$, we have $c=2$. This implies that $a=b$. Now, we assume $a=b > 2=c$. Then, we have $p_i,q_i > r_i$ for enough large $i$, and we should choose the sequence of indices as $(1,2,1,2,\dots)$ or $(2,1,2,1,\dots)$. By Proposition~\ref{prop: limit theorem}, we have $p_i=q_i=a$ and $r_i=2$ because $S_i$ converges to some nonzero number $a$. However, $S_i=(a,a,2)$ satisfies (M1), and this contradicts the fact that $S_i$ satisfies (M2). Thus, we have $a=b=c=2$. 
\\
(b)\ Note that $C(S)=C(S_i)=p_i^2+q_i^2+r_i^2-p_iq_ir_i$ does not depend on $i$. Moreover, the right hand side consists of the continuous function. Thus, by considering the limit $i \to \infty$, we have $C(S)=C(2,2,2)=4$.
\end{proof}

Based on this result, we define the following map.
\begin{dfn}\label{dfn: L function}
For any cluster-positive element $S \in \mathcal{S}^{+}$, let $L(S) \in \mathcal{S}^{+}$ as follows:
\begin{itemize}
\item If $S$ satisfies (A), let $L(S)$ be the element of $\Gamma(S)$ satisfying (M1). (By Lemma~\ref{lem: M1}, this is uniquely determined by $S$.)
\item If $S$ satisfies (B), define $L(S)=(2,2,2)$.
\end{itemize}
\end{dfn}
This $L(S)$ is like an infimum of $\Gamma(S)$ for any cluster-positive element $S \in \mathcal{S}^{+}$. Note that, for any $S$ satisfying (A), then $L(S)$ is the smallest element in $\Gamma(S)$ (Lemma~\ref{lem: M1}) and satisfies (M1). We summarize propeties of $L(S)$, which was shown in Lemma~\ref{lem: M1} and Lemma~\ref{lem: convergence}.
\begin{prop}\label{prop: properties of L}
Let $S \in \mathcal{S}^{+}$ be cluster-positive. Then, we have the following properties.
\\
(a)\ We have $C(S)=C(L(S))$.
\\
(b)\ $L(S)$ satisfies (M1).
\\
(c)\ For any $S' \in \Gamma(S)$, we have $L(S) \leq S'$.
\end{prop}

\section{Proof of Lemma~\ref{lem: 2.1}} \label{sec: proof of lemma 2.1}
Let $S \in \mathcal{S}^{+}$ be cluster-positive. Then, by Proposition~\ref{prop: properties of L}, there is $L(S)$ satisfying (M1) such that $C(S)=C(L(S))$. Thus, to show Lemma~\ref{lem: 2.1}, it suffices to show the following statement.
\begin{lem}\label{lem: 2.1.1}
If $S \in \mathcal{S}^{+}$ satisfies (M1), we have $C(S) \leq 4$.
\end{lem}
For later, we show the following lemma at the same time.
\begin{lem}[cf. {\cite[Lem.~3.3]{BBH11}}]\label{lem: 2.1.2}
Let $S=(p,q,r) \in \mathcal{S}^{+}$ satisfy (M1), and we assume $p \geq q \geq r$.
\\
(1)\ We have $C(S) \leq C\left(r,r,r\right)$.
\\
(2) The equality $C(S)=4$ holds if and only if $S=(p,p,2)$ for some $p \geq 2$.
\end{lem}

In \cite{BBH11}, they showed it for any integer triplet $p,q,r \in \mathbb{Z}$. This result is a generalization for the real case. The following proof is similar but slightly different to the one of \cite[Lem.~3.3]{BBH11}. In particular, the proof is different when we consider $2 < r <3$.

\begin{proof}
Since $S$ satisfies (M1), by Lemma~\ref{lem: M1}~(a) and Proposition~\ref{prop: charac. of M}~(1), we have $\frac{1}{2}qr \geq p \geq q \geq r \geq 2$. Now, we fix $r \geq 2$ and we see $C(S)$ as a function of $(p,q) \in \mathbb{R}^2$ satisfying the above inequalities. The region where $(p,q)$ satisfies these inequalities is illustrated in Figure~\ref{fig: region}. Set $f(p,q)=p^2+q^2+r^2-pqr$.
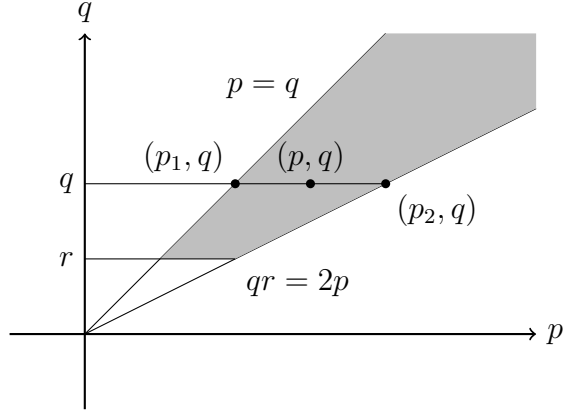
\begin{figure}[htbp]
\centering
\begin{tikzpicture}
\draw[->, thick] (-1,0)->(6,0) node [right]{$p$};
\draw[->, thick] (0,-1)->(0,4) node [above]{$q$};
\draw(0,0)--(2,1) node [below right] {$qr=2p$}--(6,3);
\draw(0,0)--(3,3) node[above left] {$p=q$}--(4,4);
\fill [lightgray] (1,1)--(2,1)--(6,3)--(6,4)--(4,4)--cycle;
\draw(0,1)node[left]{$r$}--(2,1);
\draw (0,2)node[left]{$q$}--(2,2)node[above left]{$(p_1,q)$}--(3,2) node[above]{$(p,q)$}--(4,2)node[below right]{$(p_2,q)$};
\fill (2,2)circle(0.06);
\fill (4,2)circle(0.06);
\fill (3,2)circle(0.06);
\end{tikzpicture}
\caption{The region of $(p,q)$}\label{fig: region}
\end{figure}
\par
First, we show Lemma~\ref{lem: 2.1.2}~(1). We fix $q \geq r$ and we view $f(p,q)$ as a function of $p$. Set $p_1=q$ and $p_2=\frac{1}{2}qr$. In this case, $p$ satisfies $p_1 \leq p \leq p_2$ as in Figure~\ref{fig: region}. Then, since $f(p,q)$ is a quadratic function with the positive leading coefficient $1$, we have $f(p,q) \leq \max(f(p_1,q),f(p_2,q))$.
We can show $f(p_1,q) \geq f(p_2,q)$ as follows:
\begin{equation}
\begin{aligned}
f(p_1,q)-f(p_2,q)&=(p_1)^2-(p_2)^2-(p_1qr-p_2qr)\\
&=(p_1-p_2)(p_1+p_2)-qr(p_1-p_2)\\
&=(p_1-p_2)(p_1-qr+p_2) \geq 0.
\end{aligned}
\end{equation}
The last inequality may be obtained by $p_1-p_2\leq 0$ and $p_1-qr+p_2=q-qr+\frac{1}{2}qr=-\frac{1}{2}q(r-2)\leq 0$. Thus, we have $f(p,q) \leq f(p_1,q)=f(q,q)$.
Moreover, since $f(q,q)=2q^2+r^2-q^2r=q^2(2-r)+r^2$ is a quadratic function of $q$ with the negative leading coefficient $2-r <0$ or a constant $r^2$, we have $f(q,q) \leq f(r,r)=C(r,r,r)$. (We used $q \geq r>0$.) Thus, we have $C(S) \leq C(r,r,r)$.
\par
Next, we show Lemma~\ref{lem: 2.1.1}. By Lemma~\ref{lem: 2.1.2}~(1), it suffices to show that $C(r,r,r)=3r^2-r^3 \leq 4$. It may be shown as follows:
\begin{equation}
4-(3r^2-r^3)=(r+1)(r-2)^2 \geq 0.\quad(r \geq 2)
\end{equation}
\quad Last, we show Lemma~\ref{lem: 2.1.2}~(2). Since $C(S)\leq C(r,r,r) \leq 4$, we need $C(r,r,r)=4$ for $C(S)=4$. Moreover, $C(r,r,r)=4$ turns into the equality $(r+1)(r-2)^2=0$. By $r \geq 2$, we need $r=2$ to be $C(S)=4$. In this case, we have
\begin{equation}
C(S)=p^2+q^2+4-2pqr=(p-q)^2+4.
\end{equation}
Thus, $C(S)=4$ is equivalent to $p=q \geq r=2$.
\end{proof}

\section{Integer skew-symmetrizable case}\label{sec: integer skew-symmetrizable case}
Here, we consider integer skew-symmetrizable matrices and their orbits by $\gamma_k$. For this, we consider $\Gamma$ as a subgroup of $\mathrm{Aut}(\mathcal{M})$ generated by $\gamma_1$, $\gamma_2$, and $\gamma_3$.
\par Recall that $\mathcal{M}$ is defined by the set of all $M=
\left(\begin{smallmatrix}
x & y & z\\
x' & y' & z'
\end{smallmatrix}\right)$ corresponding to a skew-symmetrizable matrix $B \in \mathrm{M}_3(\mathbb{R})$ as in Defnition~\ref{dfn: definition of sets}. In this section, we consider the following subsets.
\begin{dfn}
Let $\mathcal{M}_{\mathbb{Z}}$ (resp. $\mathcal{M}^{+}_{\mathbb{Z}}$) be the set of all $M \in \mathcal{M}$ (resp. $M \in \mathcal{M}^{+}$) whose entries are integral. Moreover, we write the set of all cluster-positive elements in $\mathcal{M}_{\mathbb{Z}}$ by $\mathcal{M}_{\mathbb{Z}}^{\mathrm{CP}}$.
\end{dfn}
For the sake of simplicity, we introduce an action of the symmetric group $\mathfrak{S}_3$ on $\mathcal{M}$ as
\begin{equation}
\sigma \left(\begin{matrix}
a_1 & a_2 & a_3\\
a'_1 & a'_2 & a'_3
\end{matrix}\right)
=
 \left(\begin{matrix}
a_{\sigma^{-1}(1)} & a_{\sigma^{-1}(2)} & a_{\sigma^{-1}(3)}\\
a'_{\sigma^{-1}(1)} & a'_{\sigma^{-1}(2)} & a'_{\sigma^{-1}(3)}
\end{matrix}\right),
\end{equation}
where $\sigma \in \mathfrak{S}_3$. We introduce another action $t: \mathcal{M} \to \mathcal{M}$ by $t\left(\begin{smallmatrix}
x & y & z\\
x' & y' & z'
\end{smallmatrix}\right)=\left(\begin{smallmatrix}
x' & y' & z'\\
x & y & z
\end{smallmatrix}\right)$. Then, every element obtained by applying these actions for $M \in \mathcal{M}$ is called a {\em permutation} of $M$. Similarly, we define permutations of $S \in \mathcal{S}$. (For $\mathcal{S}$, we set $t=\mathrm{id}$.)
\subsection{Fundamental domain}
Recall that $\mathcal{S}$ is defined by the set of all $S=(p,q,r) \in \mathbb{R}^3$. Then, there is a map $\mathrm{Sk}: \mathcal{M} \to \mathcal{S}$ with $\mathrm{Sk}\left(\begin{smallmatrix}
x & y & z\\
x' & y' & z'
\end{smallmatrix}\right)=(\varepsilon_x\sqrt{xx'},\varepsilon_y\sqrt{yy'},\varepsilon_z\sqrt{zz'})$, where $\varepsilon_a \in \{0,\pm1\}$ is the sign of $a \in \mathbb{R}$. By using this map, we define $\widehat{\mathcal{S}}=\mathrm{Sk}({\mathcal{M}}_{\mathbb{Z}})$, $\widehat{\mathcal{S}}^{+}=\mathrm{Sk}({\mathcal{M}}_{\mathbb{Z}}^{+})$, and $\widehat{\mathcal{S}}^{\mathrm{CP}}=\mathrm{Sk}({\mathcal{M}}_{\mathbb{Z}}^{\mathrm{CP}})$.
\par
First, we give an expression of $\widehat{\mathcal{S}}$.
\begin{prop}
We have
\begin{equation}
\begin{aligned}
\widehat{\mathcal{S}}&=\{(p,q,r) \in \mathcal{S} \mid p^2,q^2,r^2,pqr \in \mathbb{Z}\}.
\end{aligned}
\end{equation}
\end{prop}
\begin{proof}
The inclusion $\widehat{\mathcal{S}} \subset \{(p,q,r)\mid p^2,q^2,r^2,pqr \in \mathbb{Z}\}$ is immediately shown by the definition of $\mathrm{Sk}(M)=(p,q,r)$. Note that $pqr \in \mathbb{Z}$ follows from (\ref{eq: pqr lemma}). Let $S=(p,q,r) \in \widehat{\mathcal{S}}$. We find $M \in \mathcal{M}_{\mathbb{Z}}$ with $\mathrm{Sk}(M)=S$. If at least one entry of $S$ is zero, for example $r=0$, take $M=\left(\begin{smallmatrix}
p^2 & 1 & 0\\
1 & q^2 & 0
\end{smallmatrix}\right) \in \mathcal{M}_{\mathbb{Z}}$. Assume $p,q,r \neq 0$. Since $p^2,q^2,r^2 \in \mathbb{Z}$, we may express $p=l\sqrt{a}$, $q=m\sqrt{b}$, and $r=n\sqrt{c}$, where $l,m,n \in \mathbb{Z}\backslash\{0\}$ and $a,b,c \in \mathbb{Z}_{> 0}$ do not have any square factors. Since $pqr=lmn\sqrt{abc} \in \mathbb{Z}$ and $lmn \in \mathbb{Z}\backslash\{0\}$, we need $\sqrt{abc} \in \mathbb{Z}$. Let $p_k$ be the $k$th prime number, and consider factorizations $a=\prod p_k^{\varepsilon_{k}^{a}}$, $b=\prod p_k^{\varepsilon_{k}^{b}}$, and $c=\prod p_k^{\varepsilon_{k}^{c}}$. Since $a$, $b$, and $c$ do not have square factors, every exponent of $p_k$ should be $0$ or $1$. Moreover, since $\sqrt{abc}=\sqrt{\prod p_k^{\varepsilon_k^a+\varepsilon_k^b+\varepsilon_k^c}}$ is an integer, we need $\varepsilon_k^a+\varepsilon_k^b+\varepsilon_k^c \in 2\mathbb{Z}$ for any $k$. In particular, $(\varepsilon_k^a,\varepsilon_k^b,\varepsilon_k^c)$ is any of $(0,0,0)$, $(1,1,0)$, $(0,1,1)$, or $(1,0,1)$. Let
\begin{equation}
u=\prod_{(\varepsilon_k^a,\varepsilon_k^b)=(1,1)} p_k,\ v=\prod_{(\varepsilon_k^b,\varepsilon_k^c)=(1,1)} p_k,\ w=\prod_{(\varepsilon_k^c,\varepsilon_k^a)=(1,1)} p_k.
\end{equation}
Then, we have $a=wu$, $b=uv$, and $c=vw$. By setting
\begin{equation}
M=\left(\begin{matrix}
lu & mv & nw\\
lw & mu & nv
\end{matrix}\right) \in \mathcal{M}_{\mathbb{Z}},
\end{equation}
we have $\mathrm{Sk}(M)=S$.
\end{proof}
As Definition~\ref{dfn: A, B conditions}, every cluster-positive element $S \in \mathcal{S}$ satisfies the condition (A) or (B). However, when we consider $S \in \widehat{\mathcal{S}}$, the case (B) does not occur.
(This proposition is a generalization of \cite[Lem.~2.1~(b)]{BBH11} for the skew-symmetrizable case.)
\begin{prop}[{cf.~\cite[Lem.~2.1~(b)]{BBH11}}]\label{prop: integer lemma}
If $S \in \widehat{\mathcal{S}}^{+}$ is cluster-positive, then $S$ satisfies (A), which is defined in Definition~\ref{dfn: A, B conditions}.
\end{prop}
\begin{proof}
Suppose that there exists a monotonically decreasing sequence (\ref{eq: (M2) sequence}) and set $S_i=(p_i,q_i,r_i) \in \widehat{\mathcal{S}}^{+}$. Then, their product $p_iq_ir_i$ is strictly monotonically decreasing. Since $p_iq_ir_i \in \mathbb{Z}$, it implies that $p_iq_ir_i$ must be negative if $i$ is enough large. Thus, $S$ must not be cluster-positive.
\end{proof}
Moreover, by Lemma~\ref{lem: M1}, an element in $\Gamma(S)$ satisfying (M1) is unique for any $S \in \widehat{\mathcal{S}}^{\mathrm{CP}}$. In particular, the set
\begin{equation}
F_{\mathrm{S}}=\{S \in \widehat{\mathcal{S}}^{+} \mid \textup{$S$ satisfies (M1)}\}
\end{equation}
is a fundamental domain for the action of $\Gamma$ on $\widehat{\mathcal{S}}^{\mathrm{CP}}$ \cite[Thm.~5.1]{BBH11}. We generalize this result for the skew-symmetrizable case.
\par
For this, we introduce a partial order $\leq$ on $\mathcal{M}$ as each corresponding entry satisfies this order on $\mathbb{R}$. Then, this order is compatible for the map $\mathrm{Sk}$.
\begin{lem}\label{lem: skew-symmetrizable M1}
Let $M=\left(\begin{smallmatrix}
a_1 & a_2 & a_3\\
a'_1 & a'_2 & a'_3
\end{smallmatrix}\right) \in {\mathcal{M}}^{+}$. Then, for any $i=1,2,3$, the following three inequalities are equivalent.
\begin{itemize}
\item[(a)] $M \leq \gamma_i(M)$ on $\mathcal{M}$.
\item[(b)] $\mathrm{Sk}(M) \leq \gamma_i(\mathrm{Sk}(M))$ on $\mathcal{S}$.
\item[(c)] $a_1a_2a_3 \geq 2 a_ia'_i$.
\end{itemize}
\end{lem}
\begin{proof}
We show the claim for the case of $i=1$. The condition (a) is equivalent to the condition that both inequalities $a'_2a'_3-a_1 \geq a_1$ and $a_2a_3 - a'_1 \geq a'_1$ hold. These two inequalities are equivalent to $a_1a_2a_3 \geq 2a_1a'_1$ since $a_1a_2a_3=a'_1a'_2a'_3$ and $a_1,a'_1 > 0$. Thus, (a) and (c) are equivalent. Similarly, we can show that (b) and (c) are equivalent.
\end{proof}
\begin{dfn}
For any $M=\left(\begin{smallmatrix}
x & y & z\\
x' & y' & z'
\end{smallmatrix}\right) \in \mathcal{M}^{+}$, we say that $M$ satisfies (M1) if $\mathrm{Sk}(M)$ satisfies (M1) in terms of Definition~\ref{dfn: Mk conditions}.
\end{dfn}
By Proposition~\ref{prop: charac. of M} and Lemma~\ref{lem: skew-symmetrizable M1}, $M =\left(\begin{smallmatrix}
x & y & z\\
x' & y' & z'
\end{smallmatrix}\right) \in \mathcal{M}^{+}$ satisfies (M1) if and only if $xyz \geq 2xx',2yy',2zz'$.
\par
Now, we can generalize Lemma~\ref{lem: M1}~(c) for $\mathcal{M}$.
\begin{lem}\label{lem: increasing sequence of skew-symmetrizable}
Let $M \in \mathcal{M}_{\mathbb{Z}}$ satisfy (M1). Then, for any $M' \in \Gamma(M)$, there is a sequence
\begin{equation}\label{eq: monotonically increasing}
M=M_0 \overset{\gamma_{t_1}}{\mapsto} M_1 \overset{\gamma_{t_2}}{\mapsto} \cdots \overset{\gamma_{t_k}}{\mapsto} M_k=M'
\end{equation}
satisfying $M_{i-1} \leq M_{i}$ and $M_{i-1} \neq M_{i}$ for any $i=1,2,\dots,k$. In particular, we have $M \leq M'$ for any $M' \in \Gamma(M)$.
\end{lem}
\begin{proof}
Since $M' \in \Gamma(M)$, there is a sequence (\ref{eq: monotonically increasing}) with $M_{i-1} \neq M_{i}$ and $t_{i-1} \neq t_{i}$ for any $i=1,2,\dots,k$. We show $M_{i-1} \leq M_{i}$. Consider a sequence $S_0=\mathrm{Sk}(M_0), S_1=\mathrm{Sk}(M_1),\dots,S_{k}=\mathrm{Sk}(M_k)$. Then, by Lemma~\ref{lem: M1}~(c), we have $S_{i-1} \leq S_{i}$, and it implies that $M_{i-1} \leq M_i$ by Lemma~\ref{lem: skew-symmetrizable M1}.
\end{proof}
As a conclusion, we may give a fundamental domain for the action $\Gamma$. (The existence and the uniqueness of the minimum element was shown in \cite[Thm.~1.2]{Sev13}.)
\begin{thm}[cf. {\cite[Thm.~1.2]{Sev13}}]\label{thm: fundamental domain}
A set
\begin{equation}
\begin{aligned}
F&=\{M \in \mathcal{M}_{\mathbb{Z}}^{+} \mid \textup{$M$ satisfies (M1)}\}\\
&=\left\{\left.\left(\begin{matrix}
x & y & z\\
x' & y' & z'
\end{matrix}\right) \in {\mathcal{M}}_{\mathbb{Z}}^{+}\ \right|\ xyz \geq 2xx',2yy',2zz' \right\}
\end{aligned}
\end{equation}
is a fundamental domain for the action of $\Gamma$ on $\mathcal{M}_{\mathbb{Z}}^{\mathrm{CP}}$. (Namely, $F \subset \mathcal{M}_{\mathbb{Z}}^{\mathrm{CP}}$ holds and, for any $M \in {\mathcal{M}}_{\mathbb{Z}}^{\mathrm{CP}}$, the set $\Gamma(M) \cap F$ is a singleton.) Moreover, the element of $\Gamma(M)\cap F$ is the smallest element in $\Gamma(M)$.
\end{thm}
\begin{proof}
For any $M \in F$, $\mathrm{Sk}(M)$ is cluster-positive since it satisfies (M1). Thus, $M$ is also cluster-positive, namely, the inclusion $F \subset \mathcal{M}_{\mathbb{Z}}^{\mathrm{CP}}$ holds.
\par
Let $M \in \mathcal{M}_{\mathbb{Z}}^{\mathrm{CP}}$. We show $\Gamma(M) \cap F \neq \emptyset$. By Proposition~\ref{prop: integer lemma}, there is $S=\gamma_{t}(\mathrm{Sk}(M)) \in \Gamma(\mathrm{Sk}(M))$ satisfying (M1) after some transformation $\gamma_{t}$. Then, since $\gamma_{t}(\mathrm{Sk}(M))=\mathrm{Sk}(\gamma_t(M))$, $\gamma_t(M)$ satisfies (M1). Thus, we have $\gamma_t(M) \in \Gamma(M) \cap F \neq \emptyset$. Next, we show $|\Gamma(M) \cap F|=1$. Without loss of generality, we may suppose that $M$ satisfies (M1). Let $M' \in \Gamma(M) \cap F$. Then, by Lemma~\ref{lem: increasing sequence of skew-symmetrizable}, we have $M \leq M'$ and $M' \leq M$, and it implies that $M=M'$.
\end{proof}

\subsection{Fixed points and finite orbits}
Here, we focus on elements $M \in \mathcal{M}^{+}$ with a finite orbit $\Gamma(M)$. In \cite{FST12a,FST12b}, they gave a classification of these matrices.
\begin{lem}\label{lem: fixed point}
An element $M=\left(\begin{smallmatrix}
x & y & z\\
x' & y' & z'
\end{smallmatrix}\right) \in \mathcal{M}^{+}$ is a fixed point for the action $\Gamma$ if and only if $\mathrm{Sk}(M)=(2,2,2)$, that is,
\begin{equation}\label{eq: fixed point}
xx'=yy'=zz'=4.
\end{equation}
\end{lem}
\begin{proof}
Let $M=\left(\begin{smallmatrix}
x & y & z\\
x' & y' & z'
\end{smallmatrix}\right) \in \mathcal{M}^{+}$ be a fixed point. Then, by considering $\gamma_1(M)=M$, we have $y'z'=2x$ and $yz=2x'$. Since $xyz=x'y'z'$ and $x,x' \neq 0$, both equalities are eqivalent to $xyz=x'y'z'=2xx'$. Similarly, we have $xyz=x'y'z'=2xx'=2yy'=2zz'$. For example, by multiplying both sides of $xyz=2xx'$ and $x'y'z'=2yy'$, we have $zz'=4$. Similarly, we have $xx'=yy'=zz'=4$. Conversely, assume $xx'=yy'=zz'=4$. Then, since $(xyz)^2=(xyz)(x'y'z')=(xx')(yy')(zz')=64$, we have $xyz=x'y'z'=8$. Thus, we have $xyz=x'y'z'=2xx'$ and, by $x,x' \neq 0$, it implies $y'z'-x=x$ and $yz-x'=x'$. These two equalities imply $\gamma_1(M)=M$. By doing a similar arguments for $\gamma_2(M)$ and $\gamma_3(M)$, we may obtain the fact that $M$ is a fixed point.
\end{proof}
When we consider $M \in \mathcal{M}_{\mathbb{Z}}^{+}$, we can write all fixed points. The matrices corresponding to the following fixed points are studied in \cite{Lam16}.
\begin{prop}\label{prop: fixed point}
All fixed points of $\mathcal{M}^{+}_{\mathbb{Z}}$ for the action $\Gamma$ are
\begin{equation}\label{eq: fixed point}
\left(\begin{matrix}
2 & 2 & 2\\
2 & 2 & 2
\end{matrix}\right),
\left(\begin{matrix}
4 & 1 & 2\\
1 & 4 & 2
\end{matrix}\right),
\end{equation}
and their permutations.
\end{prop}
\begin{proof}
It follows from Proposition~\ref{prop: fixed point}. Note that integers $z,z' \in \mathbb{Z}_{>0}$ satisfying $zz'=4$ are $(z,z')=(1,4),(2,2),(4,1)$. By considering the condition $xyz=x'y'z'=8$, we may show the claim. 
\end{proof}
\begin{rem}
Consider an integer skew-symmetrizable matrix $B$. Then, Proposition~\ref{prop: fixed point} implies that, for a {\em cyclic} matrix $B$, it holds that $\mu_k(B)=-B$ for any $k=1,2,3$ if and only if $B$ corresponds to an element in (\ref{eq: fixed point}). However, since we may verify that, for any acyclic-matrix $B$ except for the zero matrix $O$, $\mu_k(B) \neq -B$ holds for some $k=1,2,3$. Thus, $\mu_k(B)=-B$ holds for any $k=1,2,3$ if and only if $B=O$ or $B$ corresponds to an element in (\ref{eq: fixed point}).
\end{rem}
The following result is a generalization of \cite[Cor.~1.3]{BBH11} for the skew-symmetrizable case.
\begin{prop}[cf.~{\cite[Cor.~1.3]{BBH11}}]\label{prop: finite type}
All cluster-positive elements $M \in \mathcal{M}^{\mathrm{CP}}_{\mathbb{Z}}$ with $|\Gamma(M)|<\infty$ are in (\ref{eq: fixed point}) and their permutations.
\end{prop}
By considering the fact of \cite[Thm.~6.1]{FST12a} and \cite[Thm.~5.13]{FST12b}, we may show it. However, we give an alternative proof by focusing on the partial order $\leq$.
\begin{proof}
Assume $|\Gamma(M)|<\infty$ and set $S = \mathrm{Sk}(M)$. Then, since the map $\mathrm{Sk}:\Gamma(M) \to \Gamma(S)$ is surjective, we have $|\Gamma(S)| \leq |\Gamma(M)| < \infty$. Thus, there is a maximal element $S' \in \Gamma(S)$. Then, since $S'$ satisfies (M1) or (M2), $S' \leq \gamma_i(S'),\gamma_{j}(S')$ must hold for at least two indices $i,j=1,2,3$. Without loss of generality, we may assume $i=1$ and $j=2$. Then, since $S'$ is maximal, it also holds that $\gamma_{1}(S'),\gamma_2(S'), \gamma_3(S') \leq S'$ and, by $S' \leq \gamma_1(S'),\gamma_2(S')$, we have $\gamma_1(S')=\gamma_2(S')=S'$. Set $S'=(p,q,r)$. The equalities $\gamma_1(S')=\gamma_2(S')=S'$ imply $p=q$ and $r=2$. Moreover, since $\gamma_3(S')=(p,p,p^2-2) \leq S'$, we need $p^2-2 \leq 2$. Since $p \geq 2$, the inequality $p^2-2 \leq 2$ implies $p=2$. Thus, a maximum element $M'$ should satisfy $\mathrm{Sk}(M')=(2,2,2)$. By Lemma~\ref{lem: fixed point}, $M'$ is a fixed point and $M'$ must be expressed as in (\ref{eq: fixed point}). Moreover, since $\Gamma(M)=\Gamma(M')=\{M'\}$, $M=M'$ holds.
\end{proof}

\subsection{Surjectivity of the Markov constants}
In \cite[Thm.~1.2~(3)]{BBH11}, they showed that there are no integral cluster-positive elements $S$ with $0<C(S)<4$. Moreover, an element $S$ satisfying (M1) and $C(S)=0$ is only $S=(3,3,3)$. On the other hand, by extending integer skew-symmetrizable matrices $\mathcal{M}_{\mathbb{Z}}^{\mathrm{CP}}$ (or $\widehat{\mathcal{S}}^{\mathrm{CP}}$), we can find some elements with $0\leq C(S)<4$ as in Example~\ref{ex: examples of 0<C<4}. More strongly, we obtain the following result.

\begin{thm}\label{thm: surjectivity}
The maps $C:{\mathcal{M}}_{\mathbb{Z}}^{\mathrm{CP}} \to \mathbb{Z}_{\leq 4}$ and $C:\widehat{\mathcal{S}}^{\mathrm{CP}} \to \mathbb{Z}_{\leq 4}$ are surjective.
\end{thm}
\begin{proof}
Recall that the map $\mathrm{Sk}:\mathcal{M} \to \mathcal{S}$ satisfies $C(\mathrm{Sk}(M))=C(M)$. It implies $C({\mathcal{M}}_{\mathbb{Z}}^{\mathrm{CP}})=C(\widehat{\mathcal{S}}^{\mathrm{CP}})$. Thus, it suffices to show that $C(\widehat{\mathcal{S}}^{\mathrm{CP}})=\mathbb{Z}_{\leq 4}$.
For any $n \in \mathbb{Z}_{\leq 4}$, set $S=(\sqrt{5(5-n)},2\sqrt{5-n},\sqrt{5}) \in \widehat{\mathcal{S}}^{\mathrm{CP}}$. Then, we obtain $C(S)=n$.
\end{proof}
\begin{rem}
Of course, this choice of $S$ is not unique. For example, we may take $S=(\sqrt{-n+9},\sqrt{-n+9},3)$ for $C(S)=n$.
\end{rem}
Next, we give all elements $S \in \widehat{\mathcal{S}}^{\mathrm{CP}}$ satisfying $0 \leq C(S)<4$ and (M1). The classification in this region is important by the following property.
\begin{prop}[{\cite[Lem.~3.3]{BBH11}}]\label{prop: acyclic C>0}
Let $S \in \mathcal{S}^{+}$ be not cluster-positive. Then, we have $C(S) \geq 0$. In particular, for any $S \in \mathcal{S}^{+}$, we have the following statements.
\\
(a)\ If $C(S)>4$, then $S$ is not cluster-positive.
\\
(b)\ If $C(S)<0$, then $S$ is cluster-positive.
\end{prop}
For any {\em non} cluster-positive element $S \in \mathcal{S}^{+}$, there is $S' \in \mathcal{S}$ such that two entries are positive and one entry is not positive. By considering $C(S')$, we have $C(S)=C(S') \geq 0$.
\par
Thanks to Proposition~\ref{prop: acyclic C>0}, when $C(S)<0$ or $C(S)>4$, then we may determine whether $S$ is cluster-positive or not by focusing on only its Markov constant. On the other hand, when $0 \leq C(S) \leq 4$, we need to consider something other than $C(S)$. In the case of $C(S)=4$, we give a condition in Lemma~\ref{lem: 2.1.2}.
\par
Later, we give a method to obtain all orbits satisfying $C(S)=C$ with $C<4$. Recall that, for any cluster-positive elemenet $S \in \widehat{\mathcal{S}}^{\mathrm{CP}}$ corresponding to an integer skew-symmetrizable matrix, its orbit $\Gamma(S)$ may be represented by the (M1) element (Theorem~\ref{thm: fundamental domain}).
\begin{lem}\label{lem: restriction of r} 
Let $S=(p,q,r) \in {\mathcal{S}}^{+}$ satisfy (M1) and $p \geq q \geq r$ and let $C \in \mathbb{Z}_{<4}$.
\\
(a)\ There exists a real number $R > 2$ satisfying $C(R,R,R)=C$. Moreover, this $R$ is unique.
\\
(b)\ If $C(S) \geq C$, then we have $r \leq R$, where $R$ is the real number satisfying $C(R,R,R)=C$.
\end{lem}
\begin{proof}
(a)\ Since $\frac{d}{dx}C(x,x,x)=3x(2-x)<0$ for $x>2$, $C(x,x,x)$ is monotonically decreasing for $x \geq 2$. Moreover, since $C(2,2,2)=4$ and $\lim_{x \to \infty}C(x,x,x)=-\infty$, there is $R > 2$ satisfying $C(R,R,R)=C$ uniquely.
\\ 
(b)\ Since $S$ satisfies (M1), we have $r \geq 2$. By Lemma~\ref{lem: 2.1.2}~(1), we have
\begin{equation}
\begin{aligned}
C(S)&\leq C(r,r,r).
\end{aligned}
\end{equation}
Since $C=C(R,R,R) \leq C(S) \leq C(r,r,r)$, we have $C(R,R,R)\leq C(r,r,r)$. Moreover, since $C(x,x,x)$ is monotonically decreasing, we have $r \leq R$. 
\end{proof}
\begin{lem}\label{lem: restriction of p}
Let $S=(p,q,r) \in \mathcal{S}^{+}$ satisfy (M1) and $p \geq q \geq r$. Let $C \in \mathbb{Z}_{<4}$. If $C= C(S)$, then we have $2<r$ anf $p \leq r \sqrt{\frac{r^2-C}{r^2-4}}$.
\end{lem}
\begin{proof}
Since $C(S)=C<4$, we have $r>2$. (If $r=2$, we have $C(S)=(p-q)^2+4 \geq 4$.)
By solving the equation $C(S)=p^2+q^2+r^2-pqr=C$ for $q$, we have $q=\frac{1}{2}(pr\pm\sqrt{p^2(r^2-4)-4(r^2-C)})$.
Note that $\sqrt{p^2(r^2-4)-4(r^2-C)}$ is a real number since $q \in \mathbb{R}$.
If we choose $q = \frac{1}{2}(pr+\sqrt{p^2(r^2-4)-4(r^2-C)})$, then we have $q \geq \frac{1}{2}pr > p$ by $r > 2$. Thus, since $q \leq p$, we should choose $q= \frac{1}{2}(pr-\sqrt{p^2(r^2-4)-4(r^2-C)})$. Moreover, since $S$ satisfies (M1), we have $q \geq \frac{2}{r}p$ by Proposition~\ref{prop: charac. of M}~(1). Thus, it holds that
\begin{equation}
\frac{1}{2}(pr-\sqrt{p^2(r^2-4)-4(r^2-C)}) \geq \frac{2}{r}p.
\end{equation}
It implies that $(\frac{r}{2}-\frac{2}{r})p \geq \frac{1}{2}\sqrt{p^2(r^2-4)-4(r^2-C)}$. Since both sides are positive, it implies that
\begin{equation}
\left(\frac{r}{2}-\frac{2}{r}\right)^2 p^2 \geq \frac{1}{4}\{p^2(r^2-4)-4(r^2-C)\},
\end{equation}
and we have
\begin{equation}
r^2-C \geq \frac{r^2-4}{r^2}p^2.
\end{equation}
Since $r^2-4, r^2 > 0$, we have $p^2\leq r^2\frac{r^2-C}{r^2-4}$. Note that $r^2-C \geq 0$ since $r^2 \geq 4 \geq C=C(S)$ by Lemma~\ref{lem: 2.1.1}. Thus, this inequality implies that $p \leq r\sqrt{\frac{r^2-C}{r^2-4}}$.
\end{proof}

\begin{prop}
Let $S=(p,q,r) \in \widehat{\mathcal{S}}^{\mathrm{CP}}$ satisfy (M1) and $p \geq q \geq r$ and $C \in \mathbb{Z}_{<4}$. Set $R > 2$ such that $C(R,R,R)=C$. If $C(S)=C$, then we have
\begin{equation}
2 < r \leq R,\quad q \leq p \leq r\sqrt{\frac{r^2-C}{r^2-4}}.
\end{equation}
In particular, for each $C \in \mathbb{Z}_{<4}$, the number of possible $S \in \widehat{\mathcal{S}}^{\mathrm{CP}}$ satisfying (M1) and $C(S)=C$ is finite.
\end{prop}

\begin{ex}\label{ex: examples of 0<C<4}
Consider an element $S=(p,q,r) \in \widehat{\mathcal{S}}^{\mathrm{CP}}$ satisfying (M1) and $0 \leq C(S) < 4$. The real number $R > 2$ satisfying $C(R,R,R) = 3R^2 - R^3=0$ is $R=3$. Thus, by Lemma~\ref{lem: restriction of r}, under the assumption $p \geq q \geq r$, we need $2< r \leq 3$. Since $r^2 \in \mathbb{Z}$, we have $r=\sqrt{5},\sqrt{6},\sqrt{7},\sqrt{8},3$. Moreover, by setting $C=C(S)$, we have $p \leq r\sqrt{\frac{r^2-C}{r^2-4}} \leq 3\times\sqrt{\frac{3^2-0}{\sqrt{5}^2-4}}=9$ by Lemma~\ref{lem: restriction of p}. In particular, by considering $p^2,q^2,r^2 \in \mathbb{Z}$, the number of possible $S$ are at most $5 \times 81^2$. (By considering other conditions such as $pqr \in \mathbb{Z}$ and $r \leq q \leq p$, the number of elements which we need to check is much smaller than $5\times 81^2$.)
We may write all cluster-positive orbits satisfying $0\leq C(S)<4$ as follows:
\begin{prop}
For any element $S \in \widehat{\mathcal{S}}$ corresponding to an integer skew-symmetrizable matrix,\\
(a) $S$ is cluster-positive with $C(S)=0$ if and only if $S$ belongs to any of the following sets or their permutations.
\begin{equation}
\Gamma(5,2\sqrt{5},\sqrt{5}),\Gamma(3\sqrt{2},2\sqrt{3},\sqrt{6}),\Gamma(4,2\sqrt{2},2\sqrt{2}),\Gamma(3,3,3).
\end{equation}
(b) $S$ is cluster-positive with $C(S)=1$ if and only if $S$ belongs to any of the following sets or their permutations.
\begin{equation}
\Gamma(2\sqrt{5},4,\sqrt{5}),\Gamma(\sqrt{15},\sqrt{10},\sqrt{6}),\Gamma(\sqrt{14},2\sqrt{2},\sqrt{7}),\Gamma(3,2\sqrt{2},2\sqrt{2}).
\end{equation}
(c) $S$ is cluster-positive with $C(S)=2$ if and only if $S$ belongs to any of the following sets or their permutations.
\begin{equation}
\Gamma(\sqrt{15},2\sqrt{3},\sqrt{5}),\Gamma(2\sqrt{3},2\sqrt{2},\sqrt{6}),\Gamma(3,\sqrt{7},\sqrt{7}).
\end{equation}
(d) $S$ is cluster-positive with $C(S)=3$ if and only if $S$ belongs to any of the following sets or their permutations.
\begin{equation}
\Gamma(\sqrt{10},2\sqrt{2},\sqrt{5}),\Gamma(3,\sqrt{6},\sqrt{6}).
\end{equation}
(e) $S$ is cluster-positive with $C(S)=4$ if and only if $S$ belongs to $\Gamma(p,p,2)$ for some $p \geq 2$ with $p^2 \in \mathbb{Z}$ or their permutations.
\end{prop}
\end{ex}

\section{The case (B)}\label{sec: the case (B)}
Recall that there are two classes (A) and (B) among cluster-positive elements in $\mathcal{S}$. (See Definition~\ref{dfn: A, B conditions}.) 
For the skew-symmetric case, one identification was given by  \cite[Rem.~4.6]{FT19}. Here, we slightly generalize it to the skew-symmetrizable case, and give another identification.
\begin{dfn}
For any cluster-positive element $M \in \mathcal{M}^{+}$, we define the class (A) and (B) by $\mathrm{Sk}(M)$ satisfies (A) and (B), respectively.
\end{dfn}
\begin{lem}
If a cluster-positive element $M \in M^{+}$ satisfies (B), the following statements hold. 
\\
(a) There exists a unique infinite sequence
\begin{equation}\label{eq: decreasing sequence of M}
M=M_0 \overset{\gamma_{t_1}}{\mapsto} M_1 \overset{\gamma_{t_2}}{\mapsto} M_2 \overset{\gamma_{t_3}}{\mapsto} \cdots
\end{equation}
such that $M_{i-1} \nleq M_{i}$ for any $i=1,2,\dots$.
\\
(b) This sequence has the limit $L(M)=\lim_{i \to \infty} M_{i}$. Moreover, $L(M)$ satisfies $\mathrm{Sk}(L(M))=(2,2,2)$. In particular, this limit $L(M)$ is a fixed point under the transformation $\Gamma$.
\end{lem}
\begin{proof}
Let $S=\mathrm{Sk}(M)$. Consider the sequence (\ref{eq: (M2) sequence}). Then, $S_{i-1} \nleq S_{i}$ holds. Set $M_0=M$ and $M_i=\gamma_{t_i}(M_{i-1})$. Then, by Lemma~\ref{lem: skew-symmetrizable M1}, this sequence also satisfies $M_{i-1} \nleq M_{i}$ as we desired. The uniqueness also follows from Lemma~\ref{lem: skew-symmetrizable M1}. (If there are another decreasing sequence, $S=\mathrm{Sk}(M)$ also has another decreasing sequence.) Thus, (a) holds. Next, we show (b). Let $M_i=\left(\begin{smallmatrix}
a_{1;i} & a_{2;i} & a_{3;i}\\
a'_{1;i} & a'_{2;i} & a'_{3;i}
\end{smallmatrix}\right)$. By the assumption, each $a_{j;i}$ and $a'_{j;i}$ ($j=1,2,3$) are monotonically decreasing for $i=0,1,2,\dots$. Moreover, since $M$ is cluster-positive, we have $a_{j;i}, a'_{j;i} \geq 0$. Thus, each component $a_{j;i},a'_{j;i}$ should converge to some $u_{i},u'_{i} \geq 0$, and we have
$L(M)=\left(\begin{matrix}
u_{1} & u_{2} & u_{3}\\
u'_{1} & u'_{2} & u'_{3}
\end{matrix}\right)$. Moreover, since $\mathrm{Sk}(S_i) \overset{i \to \infty}{\longrightarrow} (2,2,2)$ by Lemma~\ref{lem: convergence}, we have $\sqrt{a_{j;i}a'_{j;i}} \overset{i \to \infty}{\longrightarrow} 2$. This implies $\sqrt{u_ju'_j}=2$ ($j=1,2,3$) and $\mathrm{Sk}(L(M))=(2,2,2)$, which tells us that $L(M)$ is a fixed point by Lemma~\ref{lem: fixed point}.
\end{proof}
By considering Lemma~\ref{lem: convergence} and Lemma~\ref{lem: 2.1.2}, we have the following theorem.
\begin{thm}\label{thm: identification of (B)}
Let $M \in \mathcal{M}^{+}$ be cluster-positive element. 
\\
(a)\ This element $M$ satisfies (A) if and only if $C(M)<4$ or $\mathrm{Sk}(M)$ belongs to $\Gamma(p,p,2)$, $\Gamma(p,2,p)$, or $\Gamma(2,p,p)$ for some $p \geq 2$.
\\
(b)\ This element $M$ satisfies (B) if and only if $C(M)=4$ and $\mathrm{Sk}(M)$ does not belong to $\Gamma(p,p,2)$, $\Gamma(p,2,p)$, and $\Gamma(2,p,p)$ for any $p \geq 2$. Moreover, in this case, there exists a decreasing sequence (\ref{eq: decreasing sequence of M}) uniquely, and it converges to some fixed point $L(M)$.
\end{thm}
\begin{proof}
Let $S=\mathrm{Sk}(M)$.
By Lemma~\ref{lem: convergence}, if a cluster-positive element $S$ satisfies (B), then we have $C(S)=4$. Moreover, cluster-positive elements $S \in \mathcal{S}^{+}$ satisfying $C(S)=4$ and (A) should belong to $\Gamma(p,p,2)$, $\Gamma(p,2,p)$, or $\Gamma(2,p,p)$ by Lemma~\ref{lem: 2.1.2}~(2). Thus, the claim holds.
\end{proof}

\bibliographystyle{alpha}
\bibliography{bunken}
\end{document}